\documentclass[11pt]{article}
\usepackage[sort,square, comma, numbers]{natbib} % this is for automatically sorting inline citations e.g. [14,5] ---> [5,14]

\newcommand{\red}{\mathrm{red}}
\newcommand{\pres}[2]{\langle #1\:|\:#2 \rangle}

\newcommand{\EDToL}{\textsf{EDT0L }}

\newcommand{\LengthRelation}{\textsf{L}}
\newcommand{\WeightedLengthRelation}[1]{\LengthRelation_{\vec{#1}}}
\newcommand{\FreeLength}[1]{(#1^*, \cdot, 1, =, \LengthRelation)}
\newcommand{\FreeWeightedLength}[2]{(#1^*, \cdot, 1, =, \WeightedLengthRelation{#2})}
\newcommand{\FreeLengthNoFree}[1]{(#1, \cdot, 1, =, \LengthRelation)}
\newcommand{\FreeWeightedLengthNoFree}[2]{(#1, \cdot, 1, \WeightedLengthRelation{#2})}

\usepackage{pdfsync} 
\usepackage{amsmath} 
\usepackage{lmodern}
\usepackage[T1]{fontenc}
\usepackage[english]{babel}
\usepackage[utf8]{inputenc}
\usepackage{geometry}
\usepackage{amssymb}
\usepackage[hypcap]{caption}
\usepackage{amsmath, amsthm, amsfonts, amscd}
\usepackage[nottoc,numbib]{tocbibind}
\usepackage{xcolor}
\usepackage{ latexsym }
\usepackage{mathrsfs}
\usepackage{float}
\usepackage{graphicx}
\usepackage[labelformat=empty]{caption}
\usepackage{microtype}
\usepackage{todonotes}
\usepackage{verbatim}
\usepackage{hyperref}

\usepackage{mdwlist}
\usepackage{mathabx}
\usepackage{mathtools}
\usepackage[utf8]{inputenc}
\usepackage[english]{babel}
\usepackage{tikz-cd} 
\newtheorem{theorem}{\scshape{Theorem}}[section]
\newtheorem{theoremletter}{Theorem}
\newtheorem{corollaryletter}[theoremletter]{Corollary}
\newtheorem{cor}[theorem]{\scshape{Corollary}}
\newtheorem{prop}[theorem]{\scshape{Proposition}}
\newtheorem{lemma}[theorem]{\scshape{Lemma}}

\newtheorem{question}[theorem]{\scshape{Question}}

\newtheorem*{theorem*}{Theorem}

\newtheorem*{prop*}{Proposition}
\newtheorem*{cor*}{Corollary}

\theoremstyle{definition}
\newtheorem{dfn}[theorem]{\scshape{Definition}}

\newtheorem{example}[theorem]{\scshape{Example}}

\newtheorem{remark}[theorem]{\scshape{Remark}}

\def\subset{\subseteq}

\def\={\cong}

\def\mb{\mathbf}
\def\mc{\mathcal}
\def\mbb{\mathbb}

\def\rel{R}
\newcommand{\lb}{\langle}
\newcommand{\rb}{\rangle}

\title{On equations and first-order theory of one-relator monoids}
\author{Albert Garreta\footnote{Albert Garreta,
		   Department of Mathematics, University of the Basque Country, 
		   48080 Bilbao, Spain,
           \emph{albert.garreta@ehu.eus},
           },\quad
           Robert D.\ Gray\footnote{Robert D.\ Gray,
           School of Mathematics, University of East Anglia, Norwich NR4 7TJ, England, UK,
           \emph{Robert.D.Gray@uea.ac.uk}
\newline\-------------------------------------------------\newline
}}

\begin{document}

\maketitle

\begin{abstract}
 
    We investigate systems of equations and the first-order theory of one-relator monoids. We describe a family $\mc{F}$ of  one-relator monoids  of the form $\langle A\mid w=1\rangle$  where for each monoid $M$ in $\mc{F}$, the longstanding open problem of decidability of word equations with length constraints reduces to the Diophantine problem (i.e.\ decidability of systems of equations) in $M$. We achieve this result by finding an interpretation in $M$ of a free monoid, using only systems of equations  together with  length relations. It follows that each monoid in $\mc{F}$ has undecidable positive AE-theory, hence in particular it has undecidable first-order theory.  
    %}
    %
    The family $\mc{F}$ includes many one-relator monoids with torsion $\langle A\mid w^n = 1\rangle$ ($n>1$). 
    In contrast, all one-relator groups with torsion are hyperbolic, and all hyperbolic groups are known to have decidable Diophantine problem. 
    We further describe a different class of  one-relator monoids  with decidable Diophantine problem. 
\end{abstract}

\section{Introduction}\label{sec:intro}

Two important longstanding open algorithmic problems in algebra are the decidability of the conjugacy problem for one-relator groups, and of the word problem for one-relator monoids.  Each of these problems is a special case of the much more general and open question  of whether the Diophantine problem (decidability of systems of equations) is decidable in one-relator groups, or in one-relator monoids.  
A positive answer to any of these would give a positive resolution to one (or both) of the  
open questions about the word and conjugacy problems mentioned above.  On the other hand, if the Diophantine problem turns out to be undecidable for one-relator monoids or one-relator groups, then this would give a natural undecidable decision problem for these classes, which could 
lead to further undecidability 
results for these classes that would then %
have the potential to shed new light on fundamental questions like the word and conjugacy problems.  

The Diophantine problem for one-relator groups has recently received attention in the literature, with positive results obtained for solvable Baumslag-Solitar groups $BS(1,n)=\langle a, b\mid a^{-1} b a = b^n \rangle$ ($n\in\mbb{Z}$); see \cite{Kharlampovich2019}
(in the same paper the authors also solve the problem for   
wreath products of the form $A\wr \mbb{Z}$, where $A$ is a finitely generated abelian group). The aim of this paper is to initiate the study of Diophantine problems (and related model-theoretic questions) for one-relator monoids. We shall obtain both positive and negative (undecidability) results, and will also establish a close connection between these problems and the problem of solving word equations with length constraints, which is a longstanding open problem in computer science.

Our main result describes a family of one-relator monoids $\mc{F}$ such that for any $M\in \mc{F}$ it is possible to reduce decidability of word equations with length constraints ---a longstanding open problem in computer science--- to the Diophantine problem in $M$.  We further prove  decidability of the Diophantine problem for a certain class of one-relator monoids. As a corollary we obtain  undecidability of the positive $AE$-theory (hence of the first-order theory) of any one-relator monoid belonging to $\mc{F}$.  To the best of our knowledge, this provides the first examples of one-relator monoids with undecidable positive AE-theory (with coefficients), excluding the free monoid. Other examples of one-relator monoids  with undecidable first-order theory with coefficients can be found in Theorem 1 of \cite{KharlampovichLopez}.

\medskip

Equations in monoids and groups have been widely studied during the past  decades, being of interest in several areas, ranging from computer science to group  and model theory.  
For a detailed account of the history, motivation and key results in this area we refer the reader to the survey articles \cite{KhMi_rio, KhMi_icm, Diekert_1700,Romankov_survey}.  By the \emph{Diophantine problem}  we mean the algorithmic problem of determining if any given system of equations has a solution or not. 
 Two classical results due to Makanin show that  the Diophantine problem is decidable in any free monoid \cite{makanin2} and in any free group \cite{Makanin}.   Based on Makanin's algorithm, Razborov \cite{Razborov1} provided a powerful description of the sets of solutions to systems of equations in free groups via what were later called Makanin-Razaborov diagrams. This played a key part in the solution to the Tarski problems \cite{Kha_Mia_tarsi, Sela_tarski} regarding groups elementary equivalent to a free group. 
 
In subsequent years new decidability algorithms and descriptions of solutions have appeared: in \cite{Plandowski2004} Plandowski describes a polynomial space algorithm for deciding word equations based on a compression technique. In \cite{jez_linear_space} Je\.{z} shows that word equations can be solved in non-deterministic linear space, and in \cite{EDT0L} it is proved that the solution set of a word equation is an \EDToL language (in particular, it is an indexed language), furthermore this set can be computed in  polynomial space \cite{Diekert2017a}.  
More recently, in  \cite{Sela2016} Sela presents the first in a sequence of papers devoted to investigating the structure of sets of solutions to systems of equations over a free semigroup via a Makanin—Razborov diagram analogue.  Diophantine problems have been extensively considered also in different classes of groups and monoids, see e.g.\ \cite{Deis2007, Diekert2017b, lohrey, Diekert2017a, Casals2, Diekert, lohrey, GMO,  Rozenblat1985}. For us the most relevant result in this direction is the decidability of the Diophantine problem in hyperbolic groups \cite{Rips1995, dahmani}.

A variation relevant to the present paper is the problem of word equations with \emph{length constraints} (in short, WELCs). This consist of a (system of) word equation(s) together with finitely many linear inequalities involving the length of solutions (see Subsection \ref{s: preliminaries_logic} for a formal definition).  The problem of determining whether WELCs are decidable has been  open  for decades now and is of major interest in computer science. Some partial cases and variations have been successfully studied in \cite{length_ctrts, length_ctrts_2, ganesh, Buchi, ganesh_18}. 
As hinted at in \cite{ganesh}, extending word equations with constraints that involve some type of length relation or letter-counting seems to always lead to undecidability. Indeed, many problems closely related to WELCs are undecidable, as shown in some of the previous references.

WELCs are  of interest in industry where they are applied for  program verification, code debugging, security analysis, document spanning, etc. A WELC is a particular instance of a so-called Satisfibility Modulo Theory (SMT) problem, which, roughly speaking, is a satisfiability problem for a first order sentence that combines different types of formulas from different languages (such as the language of monoids, which allows to write word equations, and the language of Presburger arithmetic, which allows to write linear integer equations and inequalities). In practice, such problems are usually tackled by so-called SMT solvers, which are programs that rely on different heuristics for solving certain types of SMT problems (different SMT solvers support different possible languages and fragments of a theory). Usually,  SMT solvers are desgined to be fast and usable in real life, which in turn means that often they are not complete i.e.\ it is not guaranteed that the solver will be able to correctly solve a given input.    We refer to  \cite{ SMT_deMoura, SMT_Barrett} for further information on general SMT solvers and their applications.   There exists a variety of fast SMT
solvers  which can handle in particular word equations with rational constraints and length constraints \cite{hampi, z3str3, cvc4, ABC, stranger, norn, S3P} (we stress again that these programs are not complete, i.e.\ they cannot successfully solve any input problem). 

A further point of interest  is that WELCs are reducible to the problem of solving systems of integer-coefficient polynomial equations in $\mathbb{Z}$ \cite{matyasevich}. Thus a proof of undecidability of WELCs would provide a new solution to Hilbert's 10th Problem, which states that equations in the ring $\mathbb{Z}$ are undecidable \cite{matiasevich_H10}.

\medskip

 We would like to emphasize  how the Diophantine problem generalizes and contains many well-known and studied algorithmic problems. Notably, and as already mentioned, both the word problem and the conjugacy problem are particular cases of the Diophantine problem (see \cite{Araujo2014, Araujo2018, Narendran1984,
Narendran85, Otto1984, Zhang1991, Zhang1992} for definitions and results regarding the conjugacy problem in monoids). Moreover, the left and right divisibility problems in monoids, as well as decidability of Green's orders  $\leq_{\mathcal{R}}$ and $\leq_{\mathcal{L}}$ are  particular instances of the Diophantine problem. Thus proving that the latter is decidable in some specific group or monoid implies that any of the previously mentioned problems are decidable. Conversely, undecidability of any of the mentioned problems implies undecidability of the Diophantine problem. In a similar vein, systems of equations are particular instances of positive $AE$-formulas, which in turn  are first-order formulas. Hence similar considerations hold for the problem of decidability of the positive $AE$-theory, or of the first-order theory, of a monoid or group.

\medskip

There are several natural classes of one-relator 
monoids
for which the word problem has been shown to be decidable. Specifically, Adjan \cite{Adjan1966} showed that all one-relator monoids defined by presentations of the form $\langle A \mid w=1 \rangle$ have decidable word problem.  Monoid presentations where all of the relations are of the form $w=1$ are commonly called \emph{special} presentations.  Adjan solved the word problem for special one-relator monoids by showing that the group of units of such a monoid is a one-relator group, and then reducing the word problem of such a monoid to the word problem of its group of units. Then  decidability of the  word problem for the special one-relator monoid follows from Magnus's theorem.  
Similarly, in \cite{Zhang1991} Zhang proves that the conjugacy problem is decidable 
in the monoid $\pres{A}{w=1}$ provided it is decidable in the group of units of the monoid. 
Other results where an algorithmic problem in a monoid is reduced to the group of units can be found in \cite{Makanin1966, Zhang1992}.
These results immediately suggest the following question: \emph{Can the Diophantine problem of a special one-relator monoid be reduced to the Diophantine problem in its group of units?} Notice that by Proposition \ref{prop_hyperbolic}, a positive answer to this question would imply that the Diophantine problem is decidable in all special one-relator monoids with torsion. 
Note that it follows from the main result of \cite{lallement} that a one-relator monoid of the form $\pres{A}{u=1}$  
has torsion (that is, has a non-identity element of finite order) if and only if $u=w^k$ for some $k \geq 2$.   
Moreover, as we will prove in  this paper, a positive answer to this question would imply decidability of WELCs (see Corollary \ref{t:length:constrains:implication:intro}).

A modern approach to finite special monoid presentations using techniques from the theory of string rewriting systems is given by Zhang in \cite{Zhang1992}. Zhang’s methods will play an important role in the results we prove in this paper for special one-relator monoids.

\medskip

We shall now explain the main results of the paper in more detail. Before doing so, we  first  need to give some background notions. 

Given any one-relator monoid presentation of the form $\pres{A}{r=1}$, defining a monoid $M$, there is a unique decomposition of the 
word $r \equiv r_1 r_2 \ldots r_k$ such that each $r_i$ belongs to $A^+= A^* \setminus\{1\}$, each of the words $r_i$ represents an invertible element of $M$, and no proper non-empty prefix of $r_i$ is invertible, for all $1 \leq i \leq k$.  
The words $r_i$ $(1 \leq i \leq k)$ in this decomposition are called the \emph{minimal invertible pieces} of $r$.  Adjan \cite{Adjan1966} gives an algorithm for computing this decomposition for any one-relator special monoid. Minimal invertible pieces are a key concept for relating a special monoid with its group of units.%
The key idea used in Adjan's algorithm for computing the minimal invertible pieces is the following fact:

\bigskip

\noindent ($\dagger$) If $\alpha, \beta, \gamma \in A^*$ are words such that $\alpha \beta$ and $\beta \gamma$ both represent invertible elements of the monoid $M$ then all of the words $\alpha$, $\beta$ and $\gamma$ also represent invertible elements of $M$.

\bigskip

\noindent This is because $\alpha \beta$ being invertible implies $\beta$ is left invertible, while $\beta \gamma$ being invertible implies $\beta$ is right invertible, hence $\beta$ is invertible, from which it then quickly also follows that $\alpha$ and $\gamma$ are also invertible. 
We say that the words $\alpha \beta$ and $\beta \gamma$ \emph{overlap} in the word $\beta$. 
Adjan's algorithm begins with the defining relator word $r$
from the presentation $\pres{A}{r=1}$ 
which clearly represents an invertible element of $M$
(since $r=1$ in $M$) 
and first considers overlaps that $r$ has with itself. If there are overlaps 
then applying ($\dagger$)
this gives rise to new shorter words that we know are invertible, and then the process is repeated with these words and is iterated until no further overlaps are discovered. 
We refer the reader to \cite[Section~1]{lallement} for a detailed description of the this overlap algorithm. 
We will not need full details of the algorithm here, but we will use the key fact ($\dagger$) above about overlaps when giving examples to which our main results apply. Let us illustrate this now with an example.     
\begin{example}\label{ex:adjan:example}
Let $M$ be the one-relator monoid $\pres{a,b}{abcdcdabab=1}$. Since $ab$ is both a prefix and a suffix of the defining relator word $abcdcdabab$, applying the fact ($\dagger$) above about overlaps with $\alpha \beta \equiv  abcdcdabab \equiv \beta \gamma$ where $\beta \equiv ab$ it 
it follows 
that the words $\beta \equiv ab$, $\gamma \equiv cdcdabab$, and $\alpha \equiv abcdcdab$ are all invertible. 
Then overlapping the invertible word $ab$ with the invertible word $abcdcdab$ it follows from ($\dagger$) that $cdcdab$ and $abcdcd$ are both invertible. 
Then overlapping the invertible words $cdcdab$ and $abcdcd$ we deduce that $cd$ is invertible. 
This shows that this monoid presentation can be written as $\pres{a,b}{(ab)(cd)(cd)(ab)(ab)=1}$ 
where the parentheses indicate a decomposition of the defining relator into invertible pieces $ab$ and $cd$. Moreover, since $ab$ and $cd$ do not overlap with themselves, or each other, the Adjan algorithm will not compute any smaller invertible pieces  
and hence this is the decomposition of the relator into minimal invertible pieces. 
In particular $\{ab ,cd \}$ is the set of minimal invertible pieces of the relator in this example.  
\end{example}
For all the concrete examples of one-relator monoids that we give in this paper, the decomposition of the defining relator into minimal invertible pieces can be computed 
by repeated application of ($\dagger$) in exactly the same manner as in Example~\ref{ex:adjan:example}. 
In each case, we shall refer to this 
as the decomposition into minimal invertible pieces computed by the Adjan overlap algorithm. 
 
Given a set $S$ and a tuple of nonnegative integers $\vec{\lambda}=(\lambda_s\mid s\in S)$, by $|\cdot|_{\vec{\lambda}}$ we denote the \emph{$\vec{\lambda}$-weighted word-length} in $S^*$ defined as $$|w|_{\vec{\lambda}} =_{\text{def}}\sum_{s\in S}\lambda_s |w|_s, \quad  (w\in S^*),$$  where $|w|_s$ denotes the number of occurrences of the letter $s$ in $w$. By $\WeightedLengthRelation{\lambda}$ we denote the \emph{$\vec{\lambda}$-length relation} defined as $\WeightedLengthRelation{\lambda}(w, u)$ if and only if $|w|_{\vec{\lambda}} \leq |u|_{\vec{\lambda}}$.
Note that if $\lambda_s=1$ for all $s\in S$ then $|\cdot|_{\vec{\lambda}}$ and $\WeightedLengthRelation{\lambda}$ are just the standard word length and the  standard length relation, which we denote simply as $|\cdot|$ and $\LengthRelation$, respectively. Hence $\LengthRelation(u,v)$ holds if and only if $|u| \leq |v|$, for any two words $u,v\in S^*$. The tuple $\FreeWeightedLength{S}{\lambda}$ refers to the free monoid $S^*$ equipped with the relation $\WeightedLengthRelation{\lambda}$. This is the natural structure on which to write systems of word equations with ($\vec{\lambda}$-weighted) length constraints. See Subsection \ref{s: preliminaries_logic} for further details.

The main tool we use for reducing one problem to another is that of \emph{interpretability by systems of equations}  or by \emph{positive existential formulas} (Definition \ref{d: interpretability}). This is nothing more than the usual notion of interpretability \cite{Hodges, Marker} restricting all formulas to be systems of equations or disjunctions of systems of equations, respectively.  

Among other results, in this paper we prove the following.  

\begin{theoremletter}[Theorems \ref{thm_main} and \ref{thm_main_2}]
\label{t: main_thm_intro}
Let $M$ be the one-relator monoid $\pres{A}{r=1}$. 
Write $r \equiv r_1 r_2 \ldots r_k$ such that  $r_i \in A^+$ for all $i=1,\dots, k$, each of the words $r_i$ represents an invertible element of $M$, and no proper non-empty prefix of $r_i$ is invertible, for all $1 \leq i \leq k$.  
Set $\Delta = \{r_i \mid 1 \leq i \leq k\}$, so $\Delta$ is the set of minimal invertible pieces of the relator $r$. 
Suppose that:
\begin{enumerate}
    \item[(C1)] no word from $\Delta$ is a proper subword of any other word from $\Delta$, and
    \item[(C2)] there exist distinct words $\gamma, \delta \in \Delta$ with a common first letter $a$.
\end{enumerate}
Then there exists a free monoid $D$ of finite rank $n\geq 2$ and a tuple of positive integer weights $\vec{\lambda} = (\lambda_1, \dots, \lambda_{n})$ such that the free monoid with weighted length relation $\FreeWeightedLengthNoFree{S}{\lambda}$ is interpretable in $M$ by systems of equations.  Consequently, the problem of solving systems of word equations with weighted length constraints is reducible to the problem of solving systems of equations in $M$.

If additionally to (C1) and (C2) we have: 
\begin{enumerate}
    \item[(C3)] no word in $\Delta$ starts with $a^2$,
\end{enumerate}
then the above result holds with $\WeightedLengthRelation{\lambda}$ being the standard length relation $\LengthRelation$, i.e.\ $\LengthRelation(u,v)$ if and only if $|u|\leq |v|$, for $u,v\in D$. Consequently, in this case, the problem of solving systems of word equations with  length constraints is reducible to the problem of solving systems of equations in $M$. 
\end{theoremletter}
\begin{example}\label{ex:ThmA:examples}
We now give several examples to which Theorem~\ref{t: main_thm_intro} applies.  
For each of these examples the decomposition into minimal invertible pieces can be computed using Adjan overlap algorithm in the same way as in Example~\ref{ex:adjan:example}. 

Some examples of monoids satisfying conditions  (C1), (C2) and (C3) are $\langle a,b,c\mid (ab)(ac)(ab)=1\rangle$ and  $\langle a,b,c\mid ((ab)(ac)(ab))^n=1\rangle$ for $n\geq 1$, where we indicate the minimal invertible pieces with parentheses. 
In all these examples the set of minimal invertible pieces is $\Delta = \{ab, ac\}$. 
Indeed, considering overlaps of the defining relator word $((ab)(ac)(ab))^n$ with itself implies that $(ab)(ac)(ab)$ is invertible, and then overlapping this word with itself we deduce that $ab$ and $ac$ are both invertible. 
Since this pair of words do not overlap with themselves, or each other, it follows that these are the minimal invertible pieces. 
This set of words
$\Delta = \{ab, ac\}$
clearly satisfies conditions (C1), (C2) and (C3).   

In the two-generated case we have 
examples satisfying all of (C1), (C2) and (C3) such as  
 $$\langle a,b\mid (ababb)(abaabb)(ababb)=1\rangle$$ and $$\langle a,b\mid ((aba^nb^{n+1})(aba^{n+1}b^{n+1})(aba^nb^{n+1}))^m=1\rangle,$$ for all $n, m\geq 1$, where again we identify the decomposition into minimal invertible pieces using parentheses. 
For this second family of examples, by overlapping the relator word with itself we deduce that 
$(aba^nb^{n+1})(aba^{n+1}b^{n+1})(aba^nb^{n+1})$
is invertible and hence overlapping this word with itself we deduce that  
each of $aba^nb^{n+1}$ and 
$aba^{n+1}b^{n+1}$ is an invertible word. Since this pair of words do not overlap with themselves or with each other, it follows that the set of minimal invertible pieces for this example is $\Delta = \{ aba^nb^{n+1}, aba^{n+1}b^{n+1} \}$. 
It is then straightforward to verify that this set of words satisfies conditions (C1), (C2) and (C3).  
 
Dropping (C3) there are simpler two-generated examples which satisfy both (C1) and (C2) e.g. $\langle a,b\mid ((aab)(abb)(aab))^n=1\rangle$ ($n\geq 1$) with set of minimal invertible pieces $\{aab, abb\}$.
   As seen in these examples, the family of one-relator monoids  satisfying conditions (C1), (C2), and (C3) includes many one-relator monoids with torsion $\langle A \mid w^n=1\rangle$, $n> 1$, which  by Proposition \ref{prop_hyperbolic} 
   have hyperbolic group of units and hyperbolic undirected Cayley graph. We stress again that one-relator groups with torsion are hyperbolic and thus have decidable Diophantine problem \cite{Rips1995, dahmani}. 
\end{example}

In another direction we prove the following  result, which can be used to obtain many examples of special one-relator monoids with decidable Diophantine problem, as described in Section \ref{sec:final:section:applications:open:problems}.

\begin{theoremletter}[Theorem \ref{thm:free:prod:application}]\label{t: decidable_intro}
Let $M = \pres{A}{w=1}$ and suppose that every letter in $w$ is invertible in $M$. Let $G = \pres{B}{w=1}$ where $B \subseteq A$ is the set of letters that appear in $w$. Then $G$ is a one-relator group, and if the Diophantine problem is decidable in $G$ then it is decidable in $M$.  
\end{theoremletter}

Comparing Theorem~\ref{t: decidable_intro} with Theorem~\ref{t: main_thm_intro}, in both results we decompose the defining relator $r \equiv r_1 r_2 \ldots r_k$ into words $r_i$ that are invertible in $M$. Theorem~\ref{t: main_thm_intro} is the case where all the words $r_i$ have size one, i.e. they are single letters, in which case the theorem shows a reduction of the Diophantine problem of $M$ to its group of units.  

In Section \ref{sec:final:section:applications:open:problems} we provide some examples of monoids satisfying the hypotheses of Theorem \ref{t: decidable_intro}, as well as a list of questions and open problems.

\medskip

An immediate consequence of Theorem \ref{t: main_thm_intro} is the following
\begin{corollaryletter}
\label{t:length:constrains:implication:intro}
If word equations with length constraints are undecidable, then so is the Diophantine problem in any one-relator monoid of the form $\langle A \mid w=1\rangle$ satisfying conditions (C1), (C2) and (C3). On the other hand, proving that the Diophantine problem is decidable in some of these monoids would imply that word equations with length constraints are decidable.

In particular, if the Diophantine problem is decidable for all one-relator monoids with torsion 
$\langle A \mid w^n=1\rangle$, with $n>1$, 
then this would imply that word equations with length constraints are decidable.
\end{corollaryletter}

In addition to the Diophantine problem, we also obtain results about the decidability of the first-order theory, and more precisely of the positive $AE$-theory, of some one-relator monoids. The first-order theory with coefficients of a free nonabelian semigroup was shown to be undecidable by Quine \cite{quine} (all free structures in this paragraph are implicitly assumed to be nonabelian). Quine's result was strengthened  in \cite{Durnev, marchenkov} by proving that the positive $AE$-theory with coefficients of a free semigroup is undecidable. This contrasts with the aforementioned decidability result of Makanin for systems of equations, and also with the fact that the first-order theory of free groups is decidable as part of the solution to Tarski problems \cite{Kha_Mia_tarsi}. A consequence of Theorem \ref{t: main_thm_intro} is the following 

\begin{theoremletter}[Theorem \ref{t: undec_AE_one_relator}]\label{t: last_thm_intro}
Let $M$ be a monoid with presentation $\langle A \mid w = 1\rangle$ for some set $A$ and some word $w\in A^*$ satisfying  the conditions (C1)  and (C2) of  Theorem \ref{t: main_thm_intro}. Then the positive $AE$-theory with coefficients of $M$ is undecidable. In particular, the first-order theory with coefficients of $M$ is undecidable.
\end{theoremletter}

The paper is organized as follows: in Section \ref{s: preliminaries_logic} we provide all the necessary background regarding equations, first-order theory, and tools for obtaining reductions of one algorithmic problem to another. Section \ref{sec:DP:for:one-relator} contains the main results of the paper. Section \ref{sec:DP:for:one-relator} finishes with a small subsection where we obtain results regarding the hyperbolicity of the group of units and of the Cayley graph of some one-relator monoids. Finally, in Section \ref{sec:final:section:applications:open:problems} we present examples and  applications of our results, and we provide a list of open questions.

\section{Preliminaries}

In this section we provide the necessary background definitions and results from model and semigroup theory 
that will be needed in this article. 
In Subsections \ref{s: preliminaries_logic} and \ref{s: interpretability} we shall state the model-theoretic definitions for general structures, although throughout the paper these will be used only on monoids, or on monoids with some extra function or relation such as a length relation.  Further background on model theory can be found in \cite{Hodges, Marker}. See \cite{Arora} for notions of computational and complexity theory,  \cite{Howie} for semigroup and monoid theory background, and \cite{lyndon_schupp} for notions in combinatorial group theory.

\subsection{Equations, first-order theory, and other problems} \label{s: preliminaries_logic}

We follow Sections 1.1. and 1.3 from \cite{Hodges}. 
We fix $X$ and $A$ to denote a finite set of variables and a finite set of constants, respectively.

We describe structures by tuples $S=(U, f_1, f_2, \dots, r_1, r_2, \dots, c_1, c_2, \dots)$, where $U$ is the domain of the structure, the $f_i$ are function symbols,  the $r_i$ are relation symbols, and the $c_i$ are constant symbols. The equality relation $=$ is always assumed to be one of the relations of $S$ and  is usually omitted from the list $r_1, \dots$   The tuple $(f_1, f_2, \dots, r_1, r_2, \dots, c_1, c_2, \dots)$ is  the language (or signature) of $S$. We make the convention that this tuple is implicitly enlarged with as many  elements from $U$ as needed. These extra elements are called \emph{coefficients} (or \emph{parameters}). Sometimes we identify the whole structure with its domain. For example, we denote the free monoid generated by $A$ simply by $A^*$, omitting any reference to the concatenation operation $\cdot$ or the identity element $1$ or the equality relation $=$. 

An \emph{equation} in a structure $S$ with language  $L$ is an atomic formula in the language $L$ with coefficients. Recall that an \emph{atomic formula} is one that makes no use of quantifiers, conjunctions, disjunctions, or negations. Thus an equation in $S$ is a formula constructed using only variables, constant elements from $U$ (because we allow the use of coefficients by convention), functions $f_i$, and a single relation $r_i$.  For example if $S$ is a monoid generated by $A$ then an equation in $S$ is a formal expression  of the form $w_1(X,A) = w_2(X,A)$, where $w_1(X,A)$ and $w_2(X,A)$ are words in $(A\cup X)^*$. A \emph{solution} to such equation is a map $f: X \to S$ such that $w_1(f(X),A) = w_2(f(X),A)$ is true in $S$.   By $w_i(f(X),A)$ we refer to the word obtained from $w_i$ after replacing each variable $x\in X$ by the word $f(x)$. A \emph{system of equations} in $S$ is a conjunction of equations in $S$.  Alternatively one can define equations as formulas of the form $\exists x_1 \dots \exists x_n \phi(x_1, \dots, x_n)$ where $\phi$ is an atomic formula as above on variables $x_1, \dots, x_n$, with the relation in $\phi$ being equality. We use  these two formulations interchangeably.

Equations in a free monoid $A^*$ receive the special name of \emph{word equations}. One can consider equations in more complicated structures, such as the structure $\FreeLength{A}$ obtained from the free monoid $A^*$ (which we identify with the tuple $(A^*, \cdot, 1, =)$) by adding the length relation $\LengthRelation$ defined by the rule $\LengthRelation(u,v)$ if and only if $|u| \leq |v|$, for all $u,v\in A^*$, where $|\cdot|$ denotes  length of words and $1$ is the identity element. A system of equations in $\FreeLength{A}$ is called  a system of \emph{word equations with length constraints}. This is a system of word equations $\Sigma$ together with a finite conjunction $\mc{C}$ of formal expressions of the form $\LengthRelation(w_1,w_2)$, each called a \emph{length constraint}, where $w_1, w_2 \in (X\cup A)^*$. A map $f: X \to A^*$ is a solution to such system if it is a solution to $\Sigma$  and $|w_1(f(X),A)| \leq |w_2(f(X), A)|$ for each length constraint $\LengthRelation(w_1, w_2)$ appearing in $\mc{C}$. 

Alternatively to the length constraint one can consider the more general notion of  \emph{weighted length constraint}, which we define now. Let  $\vec{k} = (k_a\mid a\in A)$ be a tuple of natural numbers, one for each constant $a\in A$. Then by $|\cdot|_{\vec{k}}$ we denote the map $|\cdot|_{\vec{k}}: A^* \to \mbb{N}$  defined by $$|h|_{\vec{k}} = \sum_{a\in A} k_s n_a(a),$$ where $n_a(h)$ is the number of times that the letter $s$ appears in $h$. We call $|\cdot|_{\vec{k}}$ the \emph{$\vec{k}$-weighted} length function of $A^*$. We further  let $\WeightedLengthRelation{k}$ denote the relation in $A^*$ defined by the rule $\WeightedLengthRelation{k}(h, g)$ if and only if $|h|_{\vec{k}} \leq |g|_{\vec{k}}$, and call $\WeightedLengthRelation{k}$ the \emph{$\vec{k}$-weighted length relation} in $A^*$.  Note  that if $\vec{k}$ consists solely of $1$'s then $|\cdot|_{\vec{k}}$ is the usual length of words $|\cdot|$ and $\WeightedLengthRelation{k}$ is the  length relation $\LengthRelation$.

The \emph{Diophantine problem} in a structure $S$, 
denoted $\mc{D}(S)$, refers to the algorithmic problem of determining if a given system of equations in $S$ (with coefficients belonging to a fixed computable set) 
has a solution.  One says that $\mc{D}(S)$ is \emph{decidable} if there exists an algorithm (i.e.\ a Turing machine \cite{Arora}) that performs such task. 

Given two algorithmic problems $P_1$ and $P_2$, we say that $P_1$ is  \emph{reducible} to $P_2$ if there exists an algorithm that solves $P_1$ using an oracle for the problem $P_2$ (i.e.\ a black-box algorithm that `magically' solves $P_2$ ---see Definition 3.4 in \cite{Arora}).  
Thus in this case if $P_1$ is unsolvable then so is $P_2$: indeed, if $P_2$ was solvable then replacing the oracle in the definition above by an algorithm that solves $P_2$ would yield an algorithm that solves $P_1$, a contradiction.  As an example, $\mc{D}(\mbb{Z})$ is undecidable for $\mbb{Z}$ the ring of integers (this is the answer to Hilbert's 10th problem \cite{matiasevich_H10}), and hence $\mc{D}(M)$ is undecidable for any structure $M$ such that $\mc{D}(\mbb{Z})$ is  reducible to $\mc{D}(M)$.

Let $L$ be some language. A \emph{positive $AE$-sentence} in $L$ is a first-order sentence of the form 
$$
\forall x_1 \dots \forall x_n \exists y_1 \dots \exists y_m \psi(x_1, \dots, x_n, y_1, \dots, y_m)
$$
where $\psi$ is a quantifier-free formula without negations on the language $L$. The positive $AE$-\emph{theory} of a structure $S$ is  the set of all positive $AE$-sentences in the language of $S$ that are true in $S$. Analogously to the Diophantine problem, the positive $AE$-theory of $S$ is said to be \emph{decidable} if there exists and algorithm that, given a positive $AE$-sentence, decides 
whether or not it holds in $S$. 

One can generalize the notions in the paragraph above by replacing positive $AE$-sentences by any family of first-order sentences $\Phi$. In particular, if $\Phi$ is the set of all first-order sentences then one speaks of the \emph{first-order theory}, or the \emph{elementary theory}, of a structure. It is important to note that if the first-order theory is decidable then so is the Diophantine problem, the positive $AE$-theory, the positive universal theory (identity checking), etc.

\subsection{Reductions and interpretability}\label{s: interpretability}

In this subsection we introduce the notion of interpretability with respect to some class of formulas. This is a powerful tool which, in particular, implies reducibility of the decision problem for such class of formulas. It is nothing else than the classical model-theoretical notion of interpretability \cite{Hodges, Marker}, with the modification that formulas are required to be of some specific form (such as systems of equations). We follow Section 1.3 of \cite{Marker} (alternatively, see Sections 2.1 and 5.3 of \cite{Hodges}). 
\begin{dfn} \label{d: definability}
Let $M$ be a structure, $n$ a natural number, and  $\Phi$ a set of formulas in the language of $M$. A subset $S\subset M^n$ is called \emph{definable} in $M$ \emph{by formulas in $\Phi$} (in short,  $\Phi$-definable) if there exists a formula $$\Sigma_S(x_1,\ldots,x_n, y_1, \dots, y_k) \in \Phi,$$
with free variables $(x_1, \dots, x_n, y_1, \dots, y_k)=(\vec{x}, \vec{y})$,
such that for any $\vec{m} \in M^n$, one has that  $\vec{m} \in S$ if and only if there exists $\vec{y}_0 \in M^k$ such that $\Sigma_S(\vec{m}, \mb{y}_0)$ is true in $M$. In this case $\Sigma_S$ is said to \emph{define} $S$ in $M$.
\end{dfn}

We will make use of the following two classes of formulas $\Phi$:
\begin{enumerate}
    \item Systems of equations. In this case we replace the prefix $\Phi-$ by e-, speaking of \emph{e-definability.}
    \item Disjunctions of systems of equations. In this case we speak of \emph{PE-definability.} See below for an explanation of this terminology.
\end{enumerate}

\begin{remark}\label{r: PE_and_disjunctions_of_systems}
 It is well known that any disjunction of systems of equations is equivalent to a  positive existential sentence with coefficients (hence the name $PE$-definability), i.e.\ formulas that can be constructed using only existential quantifiers, conjunctions, disjunctions, variables,   and coefficents from the structure. 
\end{remark}
For example, the set  of all elements that commute with a given element $m \in M$ is defined by the equation $xm=mx$. Likewise, the set of all elements of $M$ that are squares is defined by the equation $x = y^2$.  The set of all elements that commute with $m$ or are a square is defined by the PE-formula $(xm =mx) \vee (x = y^2)$.

Observe that, by definition, e-interpretability and PE-interpretability allow the use of any coefficients in the domain of the structures at hand.

\begin{dfn}\label{d: interpretability}
Let $\mathcal{A}$ and $\mc{M}$ be two structures and let $\Phi$ be a family of formulas in the language of $\mc{M}$. Let further $A$ and $M$ be the domains of $\mc{A}$ and of $\mc{M}$, respectively. Then $\mc{A}$ is called \emph{interpretable} in $\mathcal{M}$ by formulas $\Phi$ (in short, $\Phi$-interpretable) if there exists $n\in \mathbb{N}$, a subset $S \subseteq M^n$ and a bijective\footnote{The most general formulation of interpretability uses onto maps instead of bijective maps. Since only bijective maps appear in the interpretations of this paper, we have chosen to use this more restricted version of interpretability. This is similar to the approach followed in Section 1.3 of \cite{Marker}. For  the  definition  of interpretability with onto maps see Section 5.4 of \cite{Hodges} or Section 1.3 of \cite{Marker}.} map, called \emph{interpreting} map,
$
\phi: S \to A
$,
such that:
\begin{enumerate}
\item $S$ is $\Phi$-definable in $\mathcal{M}$. 
\item  For every function $f=f(x_1, \dots, x_n)$ in the language of $\mathcal{A}$, the preimage by $\phi$ of the graph of $f$, i.e.\ the set $\{(s_1, \dots, s_k, s_{k+1}) \in S^{k+1} \mid \phi(s_{k+1}) = f(\phi(s_1), \dots, \phi(s_k))\}\subseteq M^{n(k+1)}$, is $\Phi$-definable in $\mathcal{M}$. 
\item Similarly, for every relation $r$ of $\mathcal{A}$ (including the equality relation $=$), the preimage by $\phi$ of the graph of $r$ is $\Phi$-definable in $\mathcal{M}$.
\end{enumerate}
\end{dfn}

Similarly as before, if $\Phi$ consists of all systems of equations in the language of $\mc{M}$ then we speak of e-interpretability, and if $\Phi$ consists of all disjunctions of systems of equations we speak of  $PE$-interpretability.  Note that a $PE$-interpretation is, in particular, an e-interpretation.

The next two results are fundamental and they constitute the main reason we use interpretability in this paper.   These are standard results whose proofs follow immediately from the Reduction Theorem 5.3.2 in \cite{Hodges} and Remark 3 after it (alternatively, see Lemma 2.7 of \cite{GMO_rings}).

\begin{prop}[Interpretability is transitive]\label{interpretation_transitivity}
Interpretability is a transitive relation. That is, given three structures $M_1, M_2,$ and $M_3$, if $M_1$ is e- or $PE$-interpretable in $M_2$ and $M_2$ is e or $PE$-interpretable in $M_3$, then $M_1$ is e- or $PE$-interpretable in $M_3$, respectively.
\end{prop}

\begin{prop}[Reduction of problems]\label{Diophantine_reduction}
Let $M_1$ and $M_2$ be two structures on languages $L_1$ and $L_2$, respectively. Assume $M_1$  is e-interpretable or PE-interpretable in $M_2$. Then the Diophantine problem in $M_1$ is reducible to the Diophantine problem in $M_2$. As a consequence, if the  second problem is decidable, then so is the first.

Similarly, the problem of deciding if any given first-order formula in the language $L_1$ holds in $M_1$ is reducible to the problem of deciding if any given formula in the language $L_2$ holds in $M_2$.  Consequently, if the first-order theory of $M_2$ is decidable then so is the first-order theory of $M_1$. The same statement holds when replacing first-order theory by positive $AE$-theory.
\end{prop}

\subsection{Monoid presentations and rewriting systems}\label{s: presentations}

Let $A$ be a non-empty alphabet. 
In this paper we will use $\equiv$ to denote graphical equality, 
that is, for two words $w_1, w_2\in A^*$, the expression $w_1 \equiv w_2$ means $w_1$ and $w_2$ are equal as word in $A^*$. 
A \emph{rewriting system} $\rel$ over $A$ is a subset of $A^* \times A^*$. 
We call $\pres{A}{R}$ a \emph{monoid presentation}. 
This monoid presentation is said to be finite if both $A$ and $R$ are finite, and infinite otherwise. 
The elements of $R$ are called \emph{rewrite rules} of the rewriting system, 
and they are called the \emph{defining relations} of the presentation. 
A rewrite rule $(u,v) \in R$ is often written as $u=v$ when writing the presentation $\pres{A}{R}$.  
For $u,v \in A^*$ 
we write $u \rightarrow_\rel v$ if 
there are words $\alpha, \beta \in A^*$ and a rewrite rule $(l,r)$ in $\rel$ such that $u \equiv \alpha l \beta$ and $v \equiv \alpha r \beta$. 
Let $\rightarrow_\rel^*$ denote the reflexive transitive closure of $\rightarrow_\rel$, 
and let $\leftrightarrow_\rel^*$ denote the 
reflexive transitive symmetric closure of 
$\rightarrow_\rel$.
The 
monoid defined by the presentation $\lb A \mid \rel \rb$
is the set  
$A^* /  \leftrightarrow_\rel^*$ 
of equivalence classes of the equivalence relation  
$\leftrightarrow_\rel^*$ with multiplication defined by
$(w_1 / \leftrightarrow_\rel^*)\cdot(w_2 / \leftrightarrow_\rel^*) 
= w_1w_2 / \leftrightarrow_\rel^*$ 
for all $w_1, w_2 \in A^*$. 
When the set of rewrite rules is clear from context, we shall omit the subscript $\rel$ and simply write $\rightarrow$, $\rightarrow^*$ and $\leftrightarrow^*$.
A word $u$ is called \emph{reduced} if no rewrite rule can be applied to it, that is, there is no word $v$ with $u \rightarrow v$. 
A rewriting system $\rel$ is called \emph{Noetherian} if there is no infinite chain of words $u_i \in A^*$ with $u_i \rightarrow u_{i+1}$ for all $i \geq 1$. 
The rewriting system system is called \emph{confluent} if whenever $u \rightarrow^* u_1$ and $u \rightarrow^* u_2$ there is a word $v \in A^*$ such that $u_1 \rightarrow^* v$ and $u_2 \rightarrow^* v$. 
A \emph{complete rewriting system} is one that is both 
Noetherian and confluent.
If $\rel$ is a complete rewriting system then each $\leftrightarrow^*$-class contains a unique reduced word. 
It follows that if $R$ is a complete rewriting system over an alphabet $A$ then the set of reduced words of this system provides a set   
normal forms (that is, unique representatives) for the elements of the monoid $M$ defined by the presentation $\lb A \mid \rel \rb$.
In this situation we call $\lb A \mid \rel \rb$ a \emph{complete presentation 
defining the monoid $M$}. 
We say that a word is \emph{reduced with respect the complete presentation $\pres{A}{R}$} if it is a reduced word with respect to the complete rewriting system $R$.  
If in addition either $A$ or $R$ is infinite then this is called an \emph{infinite complete presentation}. %

\section{One-relator monoids}
\label{sec:DP:for:one-relator}

Our interest in this section is in the Diophantine problem for one-relator monoids with presentation $\pres{A}{w=1}$.  
Throughout this section $M$ will denote the one-relator monoid defined by the one-relator presentation $\pres{A}{w=1}$.  
We shall see how this problem relates to other known difficult decidability problems. Before exploring those links we first observe one  situation where the Diophantine problem is decidable. 

\begin{theorem} 
\label{thm:free:prod:application}
Let $M = \pres{A}{w=1}$ and suppose that every letter in $w$ is invertible in $M$. Let $G = \pres{B}{w=1}$ where $B \subseteq A$ is the set of letters that appear in $w$. Then $G$ is a one-relator group, and if the Diophantine problem is decidable in $G$ then it is decidable in $M$.  
  \end{theorem}
\begin{proof}
The monoid $M$ is isomorphic to the moniod free product $G \ast C^*$ where $C = A \setminus B$. Both $G$ and $C^*$ satisfy Assumption~17  from \cite{Diekert} (a cancellativity condition which satisfied by any group and any free monoid) and Assumption~18  from \cite{Diekert} (decidability of the Diophantine problem).  Hence applying   \cite[Theorem~19]{Diekert} (taking $\mc{C}_\sigma$ to be just $\{U_\sigma, V_\sigma\}$) we obtain that the Diophantine problem of $M$ is decidable. 
\end{proof}

\begin{example}\label{ex:Z} 
As an easy example of an application of the previous theorem, we see that the Diophantine problem is decidable in the monoid $M = \pres{a,b,c,d}{aba=1}$.
Indeed, the monoid $G = \langle a, b \mid aba = 1\rangle$ is  the infinite cyclic group. 
To see this, since $a(ba)=1$ it follows that $a$ is right invertible in the monoid with right inverse $ba$, and since $(ab)a=1$ it follows that $a$ is left invertible with left inverse $ab$. Hence $a$ is invertible in the monoid $M$ with inverse $ab=ba$ (because $ab=ab(aba)=(aba)ba=ba$). Letting $a^{-1}$ denote the inverse of $a$ in this monoid we have $b = a^{-2}$. Hence $G$ is the infinite cyclic group generated by $a$.
The argument above shows that each letter that appears in $aba$ is invertible in $M$, and the group $G = \langle a, b \mid aba = 1\rangle$  has decidable Diophantine problem (since all free groups do). Hence the hypotheses of 
Theorem~\ref{thm:free:prod:application} are satisfied and we conclude that the Diophantine problem is decidable in the monoid $M$.  
  \end{example}
Some more complicated examples to which  Theorem~\ref{thm:free:prod:application}  applies will be discussed in Section~\ref{sec:final:section:applications:open:problems}.
  
  \medskip
  
The following lemma will be key later when studying systems of equations in some one-relator monoids (Section \ref{sec:DP:for:one-relator}). 
A definition of weighted length relation can be found in Subsection \ref{s: preliminaries_logic}.

\begin{lemma}\label{l: key_general_lemma}
Let $M$ be a monoid, let $C = \langle c_0\rangle$ be an infinite one-generated submonoid of $M$, and let $D$ be a free rank-$n$ submonoid of $M$ freely generated by a set $\{d_1, \dots, d_n\}\subseteq M$. Assume that both monoids $C$ and $D$ are e-interpretable in $M$ with interpreting map the identity map.
Assume also that for each $i= 1, \dots, n$ there exists  $k_i \in \mbb{N}$ such that $c^{k_i}d_i = 1$. Then the free monoid with weighted length relation $(D, \cdot, 1, =, \WeightedLengthRelation{k})$ is e-interpretable in $M$, where $\vec{k}=(k_1, \dots, k_n)$, and $\cdot$ is the usual concatenation operation.
\end{lemma}
\begin{proof}
Since the free monoid $D$ is e-interpretable in $M$, it suffices to show that so is the relation $\WeightedLengthRelation{k}$. Let $\Sigma_{C}(x, \vec{y})$ and $\Sigma_D(z, \vec{w})$ be two systems of equations e-interpreting $C$ and $D$ in $M$, so that an element $h \in M$ belongs to $C$ (respectively $D$) if and only if $\Sigma_C(h, \vec{y})$ (resp.\ $\Sigma_D(h, \vec{w})$) has a solution $\vec{y}_0$ (resp.\ $\vec{w}_0$) in $M$. 
 Take arbitrary elements $c\in C$ and  $d\in D$. Then $c = c_0^t$ for some $t\in \mbb{N}$, and 
$d = d_{i_1} \dots d_{i_r}$ for some $d_{i_j}$. Now, 
\begin{equation}\label{e: main_thm_cases}
cd = 
\begin{cases*}
c_0^{t-|d|_{\vec{k}}} & if $t > |d|_{\vec{k}}$,\\
1 & if $t = |d|_{\vec{k}}$,\\
d_{i_{\ell+1}} \dots d_{\ell_r} & if $t< |d|_{\vec{k}}$, \text{ and } $cd \in D$,\\
c_0^s d_{i_{\ell+1}} \dots d_{\ell_r}  & if $t< |d|_{\vec{k}}$, \text{ and } $cd \notin D$,
\end{cases*}
\end{equation}
where in the last two cases $\ell$ is the minimum  number such that $|d_{i_1} \dots d_{i_{\ell+1}}|_{\vec{k}} > t$ (we have $\ell < r$), and in the last case $s$ is some number such that $0<s< k_{i_{\ell+1}}$.
It follows that if $t\geq |d|_{\vec{k}}$ then $cd \in C$. The other implication is true as well: if we had $cd\in C$ and $t < |d|_{\vec{k}}$ then $cd = c_0^s d_{i_{\ell+1}} \dots d_{\ell_r} = c_0^{r}$ for some $r,s\geq 0$ and some $0<\ell < r$. Let $d'=d_{i_{\ell+1}} \dots d_{\ell_r}$ and let $p= |d'|_{\vec{k}}$. Note that $s< p$. Then $1=c_0^p d' = c_0^{p-s} c_0^sd' = c_0^{p-s+r}$, contradicting the assumption that $\langle c_0 \rangle $ is infinite.

We have proved that $cd\in C$ if and only if  $|c|=t \geq |d|_{\vec{k}}$. Due to the e-definability of $C$, this in turn occurs if and only if  $ \Sigma_C(cd, \vec{y})$ has a solution $\vec{y}_0$. 
Moreover, the second case of \eqref{e: main_thm_cases} and the infiniteness of $\langle c_0\rangle$ indicate that  $t=|d|_{\vec{k}}$ if and only if $cd=1$.
Hence given two elements $d_1, d_2 \in D$ we have that $|d_1|_{\vec{k}} \leq |d_2|_{\vec{k}}$ if and only if there exists an element $c\in C$ such that $cd_2 = 1$ (this ensures $|c| = |d_2|_{\vec{k}}$) and $\Sigma_C(cd_1, \vec{y})$ has a solution $\vec{y}_0$ (this ensures $|d_1|_{\vec{k}} \leq |c|$). Overall, $|d_1|_{\vec{k}} \leq |d_2|_{\vec{k}}$ if and only if the following system of equations has a solution $x_0, \vec{y}_0, \vec{z}_0$:
\begin{equation}
\begin{cases} \Sigma_C(x,\vec{y}),\\ \Sigma_C(xd_1, \vec{z}),\\ xd_2 = 1 \end{cases}
\end{equation}
It follows that the $\vec{k}$-weighted length relation $\WeightedLengthRelation{k}$ is e-interpretable in $M$.
\end{proof}

\begin{example}\label{r: n_bicyclic} The above result can be applied to the monoid with presentation \begin{equation}\label{e: n_bicyclic}\langle a, b_1, \dots, b_n \mid ab_1 = 1, ab_2=1, \ldots,  ab_n=1 \rangle,\end{equation} for any $n> 1$, thus we recover the reduction from Example 21 in \cite{Diekert}. 
\end{example}

\medskip

Let $\Delta = \{\alpha_i \; (i \in I) \} \subseteq A^+$ be the set of minimal invertible pieces of the defining relator $w$. So the word $w$ uniquely decomposes as 
\[
w \equiv \alpha_{i_1} \alpha_{i_2} \ldots \alpha_{i_k} 
\]
where each $\alpha_{i_j} \in \Delta$, and each of these words is invertible in $M$ and has no proper non-empty %%
prefix which is invertible in $M$. 
As mentioned in the introduction, we call the $\alpha_{i_j}$ the minimal invertible pieces of $w$. 
In \cite{Adjan1966} Adjan gives an algorithm for computing the minimal invertible pieces of the defining relator of a one-relator special monoid. In particular, every letter appearing in the relator represents an invertible element of the monoid if and only if all the minimal invertible pieces have size one, and this can be decided using Adjan's algorithm. Hence Adjan's algorithm can be used to test whether the hypothesis of Theorem \ref{thm:free:prod:application} above are satisfied.
This algorithm was discussed in Section~\ref{sec:intro}, see Example~\ref{ex:adjan:example} and the paragraph preceding it.  
As mentioned there, a good description of Adjan's algorithm can be found in \cite[Section~1]{lallement}. 

Since each piece $\alpha_i$ is minimal invertible, none of them is a prefix of another piece $\alpha_j$, and so $\Delta$ is a prefix code. Hence the submonoid of $A^*$ generated by $\Delta$ is free. We shall denote it $\Delta^*$. Let $B = \{b_i \mid i \in I \}$ be an alphabet in bijective correspondence with $\Delta$. Let $\phi: \Delta^* \rightarrow B^*$ be the unique homomorphism extending $\alpha_i \mapsto b_i$ for $i \in I$. It follows from Adjan's results \cite{Adjan1966} that the group of units $G$ of $M$ is isomorphic to the monoid defined by the monoid presentation  
\[
\pres{B}{\phi(w)=1} = \pres{B}{b_{i_1}b_{i_2}\ldots b_{i_k}=1}. 
\]
\begin{theorem}[\cite{Zhang1992}, Proposition~3.2]\label{thm:Zhang}
The infinite monoid presentation 
\begin{equation}\label{e: Zhang_presentation}
\pres{A}{
u=v: u, v \in \Delta^*, v <_{\mbox{\scriptsize{sh}}} u \ \ \& \ \
\phi(u) =_G \phi(v) 
}
\end{equation}
is an infinite complete presentation defining the monoid $M$.  
\end{theorem}
In the above theorem $\leq_{\mbox{\scriptsize{sh}}}$ denotes shortlex ordering, and $\phi(u) =_G \phi(v)$ means that $\phi(u)$ and $\phi(v)$ both represent the same element in the group of units $G$. For the rest of this section, when we say a word $w$ is \emph{reduced} we mean that it is reduced with respect to the above infinite complete presentation \eqref{e: Zhang_presentation}.  Our aim is to show that for a wide class of special one-relator monoids, if we could solve equations for those monoids then that would imply a solution to equation solving with length constraints in free monoids ---which is a longstanding open problem; see \cite{length_ctrts, length_ctrts_2, Buchi, ganesh, ganesh_18}. Of course, not every special one-relator monoid encodes equation solving with length constraints since, for instance, we have seen above that equations can be solved over the bicyclic  monoid. So we will need some conditions on the monoid. 
We give conditions in terms of certain combinatorial properties on the set of minimal invertible pieces $\Delta$. We suppose that the following conditions are satisfied: 
\begin{enumerate}
\item[(C1)] No word from $\Delta$ is a proper subword of any other word from $\Delta$. 
\item[(C2)] There exist distinct words $\gamma, \delta \in \Delta$ with a common initial letter $a \in A$.
\end{enumerate}

These conditions are easily satisfied and can be used to construct a wide variety of examples as we shall see in the next section. Note, for instance, if all the words from $\Delta$ have the same length, then condition (C1) will be satisfied. In particular there are one-relator monoids with torsion whose minimal invertible pieces satisfy these properties. A concrete example is given by the family monoids 
\[
\pres{a,b,c}{((ab)(ac)(ab))^k=1},  
\]
for $k > 1$ where, as we already proved in Example~\ref{ex:ThmA:examples} above, the set of minimal invertible pieces is $\{ab, ac\}$. This gives many examples of special one-relator monoids with hyperbolic undirected Cayley graphs which satisfy the conditions (C1)-(C2). Applications to examples like this will be discussed below. 

\begin{quote}
\emph{For the rest of this section let $M$ be the one-relator monoid defined by the monoid presentation 
\[
\pres{A}{r=1}
\]
where we suppose that conditions (C1)-(C2) are satisfied, and we let $a$ be a common initial letter of two distinct words from $\Delta$ }
\end{quote}

Throughout the rest of the section we denote the projection of a word $w\in A^*$ onto $M$ by $[w]$, so $[w]$ is the element of $M$ represented by the word $w$.

We now give a series of important technical lemmas. We begin with the following observation:
\begin{remark}\label{r: a_is_not_invertible}
The letter $a$ us not invertible in $M$. Indeed, there are two distinct words $\delta, \gamma \in \Delta$ having $a$ as their first letter. Since $\delta$ and $\gamma$ are subwords of $w$ which are invertible in $M$, and they minimal with this property, if $a$ was invertible then we would have $a=\gamma=\delta$. This would contradict the fact that $\gamma$ and $\delta$ are distinct words. Hence, $a$ cannot be invertible.
\end{remark}

\begin{lemma}\label{lem:C3_NEW} 
Suppose that (C1) and (C2) are both satisfied, and let $a$ be a common initial letter of two distinct words from $\Delta$. Then for every reduced word $w \in A^*$, and every positive integer $i >0$,  if $a^i w=1$ then $w$ has no prefix in $\Delta$.  
\end{lemma}
\begin{proof}
Since $a^iw=1$ it follows that $a^iw$ is not reduced and since $w$ is assumed to be reduced it follows that we can write  
\[
a^iw \equiv a^j \alpha_1 \ldots \alpha_t w''
\]
where $0 \leq j < i$, $w''$ is a suffix of $w$,    $\alpha_1 \ldots \alpha_k$ is the left hand side of a rewrite rule from \eqref{e: Zhang_presentation}, and each $\alpha_i \in \Delta$. Since $a$ is not invertible and $w$ is reduced, we have $\alpha_1 \equiv a^k w'$ where $k=i-j>0$, and  $w'$ is a non-empty prefix of $w$. Suppose, seeking a contradiction, that $w \equiv \beta w_2$ with $\beta \in \Delta$.  Note that since $w’$ is a suffix of $\alpha_1$, where $\alpha_1$ is invertible, it follows that $w’$ is left invertible. Now, if $w'$ were a prefix of $\beta$ it would follow that $w’$ is also right invertible and hence invertible.  But then since $\alpha_1$ and $w’$ are both invertible it would follow that $a^{k}$ is invertible and hence $a$ is invertible,  which is a contradiction by Remark \ref{r: a_is_not_invertible}. Therefore we must have that $\beta$ is a prefix of $w'$, but then $\beta \in \Delta$ is a proper subword of $\alpha_1 \equiv a^kw' \in \Delta$, and this contradicts (C1). This completes the proof of the lemma. 
\end{proof}

\begin{lemma}[\cite{Zhang1992}, Lemma 3.1 and Lemma 3.6]\label{lem_Z}
If $u_1, u_2 \in \Delta^*$ then $[u_1] = [u_2]$ in $M$ if and only if $[\phi(u_1)] = [\phi(u_2)]$ in the one-relator group $G$. 
\end{lemma}

\begin{lemma}\label{lem_different}
Let $\delta$ and $\gamma$ be two distinct words in $\Delta$. Then $[\delta] \neq [\gamma]$ in $M$. 
\end{lemma}

\begin{proof}
Since the words $\delta$ and $\gamma$ are distinct it follows that $\phi(\delta)$ and $\phi(\gamma)$ are distinct letters of $B$. 
This implies that $|B| \geq 2$.   
If $|B| \geq 3$ then it follows from Magnus' Freiheitssatz 
\cite[Theorem~5.1]{lyndon_schupp}
that $\phi(\delta)$ and $\phi(\gamma)$ represent distinct elements of the group $G$ and hence $[\delta] \neq [\gamma]$ in $M$, by the previous Lemma \ref{lem_Z}.  

Now suppose that $|B|=2$.  
Set $c = \phi(\delta)$ and $d=\phi(\gamma)$. 
If $c=d$ in $G$ then $cd^{-1}=1$ in $G$.    
Since $|B|=2$ and $c=d$  it follows that 
the group $G$ has torsion, and hence  
by \cite[Theorem~5.2]{lyndon_schupp}
the defining relator in the presentation of $G$ must be a proper power. Then it follows from Newman's spelling theorem 
\cite[Theorem~5.5]{lyndon_schupp}
that $cd^{-1}$ contains a subword of the defining relator (which uses no inverse of $c$ or $d$), or the inverse of such a subword, with length at least $2$.
This is clearly impossible and thus completes the proof. 
\end{proof}

\begin{lemma}[\cite{Zhang1992}, Lemma~3.3]\label{lem_Z2}
Let $u \in A^*$ be reduced. If $[u]$ is invertible then $u \in \Delta^*$.
\end{lemma}

We are interested in right inverses of powers of the element $a$. These elements clearly form a submonoid of $M$. The following result shows that the set of reduced words representing elements in this submonoid themselves form a submonoid of the free monoid $A^*$.  

\begin{lemma}\label{lem:reduced}
 Let $i, j \in \mathbb{N}$. Let $u, v \in A^*$ be reduced words such that $a^i u=1$ and $a^j v=1$. Then $uv$ is a reduced word such that $a^{i+j} uv =1$.  
\end{lemma}
\begin{proof}
 We just need to prove that $uv$ is a reduced word.  By  Lemma~\ref{lem:C3_NEW} the word $v$ does not have any prefix in $\Delta$.  If $uv$ were reducible then it would follow that there is a non-empty suffix $u_1$ of $u$, and a non-empty prefix $v_1$ of $v$, such that $u_1 v_1 \in \Delta$.  But then $u_1$ is left invertible, since $u$ is left invertible, and right invertible, since $u_1 v_1$ is right invertible.  This contradicts $u_1 v_1 \in \Delta$.
\end{proof}

Let $F$ be the set of all reduced words $\beta$ such that $a^i \beta=1$ for some $i \in \mathbb{N}$ with $i>0$, together with the empty word $1$. 
\begin{remark}\label{r: clarification_remark_about_F}
Note that by definition all the words in $F\subseteq A^*$ are reduced words with respect to the complete presentation for $M$ defined in Theorem~\ref{thm:Zhang}.
It follows from Lemma~\ref{lem:reduced} that for any words $w_1$, $w_2$ from the set $F$ the concatenation of these two words $w_1w_2$ is again a word in the set $F$ and hence in particular $w_1w_2$ is again a 
reduced word  (since all the words in $F$ are reduced words). 
Therefore, $F$ is a submonoid of the free monoid $A^*$, and all of the words in $F$ are reduced words with respect to the complete presentation for $M$.
\end{remark}

We shall now  prove that $F$ is a free submonoid of $A^*$. 
For this it will be useful to recall some standard results about submonoids of free monoids. 
Recall from \cite[Subsection 1.2]{lothaire} that given a submonoid $P$ of $A^*$ there is a unique set $\mathcal{B}$ that generates $P$ and is minimal with respect to set-theoretic inclusion; it is the set
\[
(P\setminus\{1\}) \setminus (P\setminus\{1\})^2. 
\]

The following nice characterisation of free subsemigroups of free semigroups, from Lothaire, will be useful for us; see \cite[Proposition~1.2.3]{lothaire}. 

\begin{lemma}\label{lem_lothaire} 
A submonoid $P$ of $A^*$ is free if and only if for any word $w \in A^*$, one has $w \in P$ whenever there exist $p, q \in P$ such that 
\[
pw, wq \in P. 
\]  
\end{lemma}

\begin{lemma}\label{l: F_is_free}
  $F$ is a free submonoid of $A^*$. 
\end{lemma}
\begin{proof}
Suppose that $w\in A^*$ is such that there exist $p, q \in F$ such that $pw, wq\in F$.   By Lemma \ref{lem_lothaire}, we need to show that $w \in F$.  Since $w$ is a subword of a reduced word (for example, $pw$), it is reduced. By assumption there are $i,j \geq 1$ such that $a^i p=1$ and $a^j pw =1$.  If $i=j$ then $w=1$ and since $w$ is reduced it is the empty word and this belongs to $F$.  If $i<j$ then $a^{j-i}w=1$ and so $w \in F$.  Otherwise, if $i>j$ then it would follow that $w=a^k$ for some $k>0$.  But then $a^j p w=1$ implies $a^j p a^k=1$.  This last equality implies that $a$ is invertible, contradicting (C2) and the definition of $\Delta$.  In all cases $w \in F$ so this completes the proof of the lemma.  
\end{proof} 
 
\begin{lemma}\label{lem_longestInvertPrefix}
Let $w \in A^*$ be arbitrary. Write $w \equiv w_1 w_2$ where $w_1$ is the longest prefix of $w$ which is invertible. Suppose that $w'$ may be obtained from $w$ by a single application of a relation from the presentation. Write $w' \equiv w_1' w_2'$ where $w_1'$ is the longest invertible prefix of $w'$. Then $w_1 = w_1'$ in $M$. This implies that for any pair of words $u$, $v$, if $u=v$ in $M$ then the longest invertible prefix of $u$ is equal to $1$ in $M$  if and only if the longest invertible prefix of $v$ is equal to $1$ in $M$. \end{lemma}
\begin{proof}
We consider where the relation is applied to the word  $w \equiv w_1 w_2$. If the relation is applied within either $w_1$ or $w_2$ the result is immediate, so suppose otherwise. Let $\delta_1 \ldots \delta_m \in \Delta^*$ be the subword of $w$ to which the relation is being applied. If there is a non-empty suffix $u_1$ of $w_1$, and a non-empty prefix $u_2$  of $w_2$ such that $u_1u_2 \equiv \delta_r$ for some $r$, then since $u_1$ is left invertible since it is a suffix of $w_1$, and $u_1$ is right invertible since it is a prefix of $\delta_r$, it would follow that $u_1$ is invertible, which would contradict the fact that $\delta_r$ has no proper prefix which is invertible. So we must have $w_1 \equiv \alpha \delta_{1} \ldots \delta_r$, and $w_2 \equiv \delta_{r+1} \ldots \delta_m \beta$, but then $w_1 \delta_{1} \ldots \delta_m$ is a prefix of $w$ which is invertible and is longer than $w_1$, contradicting the definition of $w_1$.  
\end{proof}

Let $m \in \mathbb{N}$ be the maximum value $m$  such that there is a minimal invertible piece $\alpha \in \Delta$ such that $a^m$ is a prefix of $\alpha$.
We define a finite set of words $X$ in the following way.
For each $1 \leq j \leq m$ and for every piece $a^j \beta \in \Delta$ (where $\beta$ might begin with $a$) let $\eta$ be the reduced word representing the inverse of $a^j \beta$ and add the word $\beta \eta$ to the set $X$. 

\begin{lemma}\label{lem:Xwordsreduced}
Every word in the set $X$ is reduced.  
\end{lemma}
\begin{proof} 
Let $a^j \beta \in \Delta$ and let $\eta$ be a reduced word  representing the inverse of $a^j \beta$.  
We claim that $\beta \eta$ is a reduced word as a consequence of assumption (C1).  Indeed, suppose for a contradiction that $\beta \eta$ is not reduced. 
It follows from Lemma~\ref{lem_Z2} that $\eta \in \Delta^*$.  
Then there is a rewrite rule from \eqref{e: Zhang_presentation} which can be applied to the word $\beta \eta$. 
Let $\lambda$ be the left hand side of such a rule noting that $\lambda \in \Delta^+$.    
Since $\beta$ and $\eta$ are both reduced words we can write $\lambda \equiv \beta_2 \eta_1$ where $\beta_2$ and $\eta_1$ are both non-empty, with $\beta \equiv \beta_1 \beta_2$ and $\eta \equiv \eta_1 \eta_2$. 
Let $\alpha_1 \in \Delta$ be the prefix of $\lambda$ which belongs to $\Delta$.         
Let $\alpha_2 \in \Delta$ be the prefix of $\eta$ which belongs to $\Delta$.    
Since $\alpha_1$ cannot be a subword of $\beta$ since by (C1) it is not a subword of $a^j \beta \in \Delta$ it follows that $\alpha_1 \equiv \alpha_1' \alpha_1''$ where $\alpha_1''$ is a non-empty prefix of $\eta$. 
But since $\eta$ is invertible this would imply that $\alpha_1''$ is invertible and thus $\alpha_1'$ is invertible, contradicting the fact that $\alpha_1 \in \Delta$ is a minimal invertible piece. 
This is a contradiction, and we conclude that $\beta \eta$ is indeed a reduced word.            
\end{proof}

Thus $X$ is a finite set of reduced words, each of which is the right inverse of some $a^j$ with $1 \leq j \leq m$.
Note also that $X$ is a finite subset of the free monoid $F$.

\begin{lemma}\label{lem:fundamental:weight}
Let $i \in \mathbb{N}$ and $w \in A^*$ be a reduced word such that $a^iw=1$ in $M$. Then there is an integer $0 < j \leq i$, with $j \leq m$,  and a non-empty prefix $w_1$ of $w$ such that $w_1 \in X$ and $a^jw_1=1$ in $M$. Moreover, with the same value of $j$, there is a decomposition 
\[
a^iw \equiv a^ka^jw'w''
\]
where $k+j=i$, $w \equiv w'w''$ and $a^jw' \in \Delta$.
In particular, if no word in $\Delta$ begins with $a^2$ then 
$w$ can be written as $w \equiv w_1 w_2 \ldots w_i$ such that
$aw_l=1$ for all $1 \leq l \leq i$. 
\end{lemma}
\begin{proof}
Let $i \in \mathbb{N}$ and $w \in A^*$ be a reduced word such that $a^iw=1$ in $M$. Since $a^iw$ is not reduced it follows 
that the left hand side $\lambda$ of one of the relations from 
\eqref{e: Zhang_presentation}
arises as a subword of $a^iw$. 
In particular $\lambda$ is a non-empty word with $\lambda \in \Delta^*$.    
Since $a$ is not invertible, no word from $\Delta$ is a subword of $a^i$, and since $w$ is reduced, $\lambda$ is not a subword of $w$. 
It follows that there is a prefix $\lambda'$ of $\lambda$ such that, $\lambda' \equiv a^jw' \in \Delta$ 
with $j>0$ and where $w'$ is a non-empty prefix of $w$.   
Thus we have the decomposition           
\[
a^iw \equiv a^ka^jw'w''
\]
where $k+j=i$, $w \equiv w'w''$ and $a^jw' \in \Delta$. 

If $k=0$ then $i=j$ and $a^iw \equiv a^jw = 1$.  So we can write $a^jw \equiv (a^jw')(w'')$ and since $(a^jw')(w'')=1$ it follows that in $M$ we have $w \equiv \beta \eta$ where  $\beta \equiv w'$, $\eta \equiv w''$, where $\eta$ is equal to the inverse of $a^j \beta$ in $M$ (note $a^j \beta$ is invertible because it belongs to $\Delta$).  Thus in this case the reduced word $w$ belongs to the set $X$, as required.

Now suppose that $k>0$. Consider the longest invertible prefix of the word $a^jw$. It is certainly non-empty since $a^jw'$ is invertible. Set $v \equiv \mathrm{red}(a^jw)$.  Then we have $a^k v =1$ with $k>0$ and $v$ a reduced word. It follows from Lemma~\ref{lem:C3_NEW} that $v$ cannot begin with a word from $\Delta$. 
Hence $v$ has no invertible prefix. Now by the last part of Lemma~\ref{lem_longestInvertPrefix}, since $v= a^j w$ in $M$, it follows that the longest invertible prefix $p$ of $a^jw$ is  equal to $1$ in $M$. So now we can write 
\[
a^iw \equiv a^ka^jw_1w_2
\]
where $k+j=i$, $w \equiv w_1w_2$ and $p\equiv a^jw_1=1$ in $M$, and $a^jw_1$ has prefix $a^jw' \in \Delta$. It then follows that in $M$ we have $w_1 = \beta \eta$ where $\beta \equiv w'$ and $\eta$ is equal to the inverse of $a^jw'$ in $M$. Also, $w_1$ is a reduced word because $w$ is reduced. It follows that $w_1 \in X$, as required. This completes the proof of the lemma.
\end{proof}

\begin{lemma}\label{lem:fin_gen_set_for_F}
$X$ is a finite generating set for the monoid $F$.  
\end{lemma}
\begin{proof}
Let $i \in \mathbb{N}$ and $w \in A^*$ be a reduced word such that $a^iw=1$ in $M$.
It follows from Lemma~\ref{lem:fundamental:weight} that   there is an integer $0 < j \leq i$, with $j \leq m$,  and a non-empty prefix $w_1$ of $w$ such that $w_1 \in X$ and $a^jw_1=1$ in $M$.
 The lemma now follows by induction.
\end{proof}

Let $\mathcal{B}$ be the unique subset of $F$ that generates $F$ and is minimal with respect to set-theoretic inclusion, that is $\mathcal{B}$ is equal to the set 
\[
(F \setminus \{1\}) \setminus (F \setminus \{1\})^2.
\]

Since $X \subseteq F$ is a finite generating set for $F$ it follows that $\mathcal{B} \subseteq X$.    

\begin{lemma}\label{lem:basis:size}
The basis $\mathcal{B}$ has size at least two. Thus the submonoid $F$ of $A^*$ is a free monoid of rank at least two.   
\end{lemma}
\begin{proof}
By assumption (C2) there are distinct words $\gamma, \delta \in \Delta$ with common initial letter $a \in A$. 
Write $\gamma \equiv a \gamma'$ and $\delta \equiv a \delta'$. 
Note that either $\gamma'$ or $\delta'$ can begin with the letter $a$.   
By Lemma~\ref{lem_different} the words $\gamma$ and $\delta$ represent different elements of the monoid $M$.  
This in turn implies that $[\gamma'] \neq [\delta']$. 
Let $(a\gamma')^{-1}$ be a reduced word representing the inverse of $a\gamma'$ in $M$, and let $(a\delta')^{-1}$ be a reduced word representing the inverse of $a\delta'$ in $M$. 
In particular $(a\gamma')^{-1},  (a\delta')^{-1} \in \Delta^*$. 
Then by definition we have $\gamma_2 \equiv \gamma' (a \gamma')^{-1} \in X$ and  $\delta_2 \equiv \delta' (a \delta')^{-1} \in X$, and both of these words are reduced words.           
Suppose, seeking a contradiction, that $F$ is a free monoid of rank $1$.    
It follows that there is a word $\nu \in A^+$ such that each of $\gamma_2$ and $\delta_2$ is, in $A^+$, equal to some power of the word $\nu$. 
But this would imply that $\gamma'$ is a prefix of $\delta'$, or vice versa. 
Suppose without loss of generality $\gamma'$ is a proper prefix of $\delta'$.    
Then $a\gamma'$ is a proper prefix of $a\delta'$. 
But this contradicts condition (C1) since both of these words belong to $\Delta$.    
This completes the proof of the lemma. 
\end{proof}

The free submonoid $F$ of the free monoid $A^*$ defined above may also naturally be viewed as a free submonoid of the monoid $M$. 
This is because, as explained in Remark~\ref{r: clarification_remark_about_F}, all the words in $F$ are reduced words and the concatenation of any two words from $F$ is again a reduced word.   
In particular, 
since distinct reduced words represent distinct elements of $M$,
the map $[\cdot]: A^*\to M$ defined by $w \mapsto [w]$ induces an embedding $[\cdot]: F \hookrightarrow M$. 
Thus we have identified a free submonoid of $M$ of rank at least two, namely the image $[F]$ of $F$ under this embedding.  

\begin{lemma}\label{lem:unique}
Let $w \in A^*$ be a word. If $a^iw=1$ and $a^jw=1$ then $i=j$.     
\end{lemma}
\begin{proof}
Seeking a contradiction suppose that $a^iw = a^jw = 1$ with $j<i$. Then $a^{i-j} = a^{i-j} a^jw = a^i w=1$. But this contradicts the fact that $a$ is not invertible.     
\end{proof}

Define a mapping $\omega: F \rightarrow \mathbb{Z}^{\geq 1} $ where $w \mapsto i$ if and only if $a^i w=1$. This is a well-defined mapping by the previous lemma. Also, it is easy to see that $\omega$ is a homomorphism to $(\mathbb{Z},+)$. The mapping $\omega$ assigns a weight to every element of the free monoid $F$.   Abusing the notation, we also use  $\omega$ to denote the map $\omega: [F] \rightarrow  \mathbb{Z}^{\geq 1}$ defined by $\omega([w])=\omega(w)$. This is well defined by the comments preceding Lemma \ref{lem:unique}.

The following result is now an immediate consequence of the previous results proved in this section.  

\begin{lemma}\label{lem_InverOfPower}
Let $w\in A^*$ be a non-empty reduced word with $w\in F$, and suppose that $a^iw=1$ with $i \geq 1$. Then the word $w$ can be written uniquely as 
\[
w \equiv w_1 w_2 \ldots w_k
\]
where $w_j \in \mathcal{B}$ for all $1 \leq j \leq k$, and 
\[
\omega(w_1) +\omega(w_2) + \ldots + \omega(w_k) = i.    
\]
In the special case that $\Delta$ contains no word beginning with $a^2$ then $\omega(w_j) = 1$ for all $1\leq j \leq k$, i.e.\ the statement above holds with $k=i$. 
\end{lemma}

Note that in particular condition (C1) is satisfied if all the pieces have the same length. We note that Adjan \cite{Adjan1966} gives an algorithm for computing the set $\Delta$ by analysing overlaps of the relator with itself.  

The following lemma will allow us to express membership in $\{a \}^*$ in terms of equations. 

\begin{lemma}\label{l:e_interp_a*}
Let $u \in A^*$ be reduced. 
Then $u \in \{a \}^*$ if and only if 
$[ua]=[au]$ in $M$. 
\end{lemma} 
\begin{proof}
Clearly if $u \in \{a\}^*$ then 
$[ua]=[au]$ in $M$. 

For the converse, suppose that $u\in A^*$ is such that $[ua]=[au]$ in $M$.
Since $a$ is right invertible and $a$ is not invertible, it follows that for all $\delta \in \Delta$ the last letter of $\delta$ is not equal to $a$.
(Note this is true for all $\delta \in \Delta$ including those $\delta$ in $\Delta$ where $\delta$ does not begin with the letter $a$.)

Seeking a contradiction, suppose that $u \not\in \{a \}^*$ and write $u \equiv u_1 a^y$ where $u_1 \in A^+$ and the last letter of $u_1$ is not equal to $a$, and $y\geq 0$.  
Consider $\red(ua) = \red(u_1 a^{y+1})$. Since for every rewrite rule $\alpha = \beta$ from \eqref{e: Zhang_presentation} neither $\alpha$ nor $\beta$ ends in the letter $a$, it follows that $\red(ua) \equiv w_1a^{y+1}$ where $w_1$ does not end in the letter $a$. 

In contrast, consider $\red(au) = \red(au_1 a^{y})$. Reasoning in the same way as in the previous paragraph $\red(au) \equiv w_2 a^y$ where $w_2$ does not end in the letter $a$ (note it may start with the letter $a$). In particular this implies that $\red(ua) \not\equiv \red(au)$ which implies $[ua] \neq [au]$. This contradicts our original assumption, and completes the proof of the lemma. 
\end{proof}

The main result we shall prove in this section is the following.

\begin{theorem}\label{thm_main}
Let $M = \pres{A}{r=1}$ and let $\Delta \subseteq A^*$ be the set of minimal invertible pieces of $r$. Suppose that:
\begin{enumerate}
    \item[(C1)] no word from $\Delta$ is a proper subword of any other word from $\Delta$, and
    \item[(C2)] there exist distinct words $\gamma, \delta \in \Delta$ with a common first letter.
\end{enumerate}
Then there exists a free submonoid $D$ of $M$ of finite rank $n\geq 2$ and a tuple of weights $\vec{\lambda} = (\lambda_1, \dots, \lambda_{n})$ such that the free monoid with weighted length relation $\FreeWeightedLengthNoFree{D}{\lambda}$ is interpretable in $M$ by systems of equations and one coefficient. 
\end{theorem}
\begin{proof}
Let $F$ be the free monoid from our previous arguments (see Remark \ref{r: clarification_remark_about_F}), and consider the embedding $[F]$ of $F$  in $M$ via the map $[]:A^* \to M$ (see the discussion above Lemma \ref{lem:unique}).  Let  $\vec{\omega}$ be the tuple  $(\omega(w_1), \dots, \omega(w_n))$  where $w_1, \dots, w_n$  freely generate $F$ and $\omega: F \to \mbb{Z}^{\geq 1}$ is the homomorphism defined after Lemma \ref{lem:unique}, so that $a^{\omega(w_i)}w_i = 1$ for all $i$ (by Lemma \ref{lem_InverOfPower}).  
Recall that $\omega$ also then defines a map $\omega:[F] \rightarrow \mbb{Z}^{\geq 1}$.   

We claim that $a$ generates an infinite submonoid of $M$. Indeed, if it did not, we would have $a^k = a^{k+\ell}$ for some $k,\ell \geq 0$. Since $a$ is right invertible (due to condition (C2)), this implies that  $a^{\ell} = 1$, from where it follows that $a$ is invertible, a contradiction.  This proves the claim. 

By Lemma~\ref{l:e_interp_a*} the submonoid $\langle a \rangle$ is interpretable in $M$ by the equation $ax = xa$ (Lemma \ref{l:e_interp_a*}). Since $[F]=\{ x\in M\mid a^tx = 1 \text{ for some } t\in \mbb{N}\setminus\{0\}\}$, it follows that $[F]$ is e-interpretable in $M$ by the system of two equations $ay=ya, yx = 1$. 
\end{proof}

The following two results follow immediately from the above Theorem \ref{thm_main} and from Proposition \ref{Diophantine_reduction} regarding reducibility of decision problems. 

\begin{cor}\label{c: reduction_word_equations_weighted_length_constraints}
Let $M$ be a monoid satisfying the hypothesis of Theorem \ref{thm_main}. Then there exists a free monoid with a weighted length relation $\FreeWeightedLengthNoFree{D}{\omega}$ such that the Diophantine problem in $\FreeWeightedLengthNoFree{D}{\omega}$ is reducible to the Diophantine problem in $M$. In particular, if the latter is decidable, then systems of word equations with $\vec{\omega}$-weighted length constraints are decidable as well.
\end{cor}

\begin{theorem}\label{t: undec_AE_one_relator}
  Any one-relator monoid of the form $\langle A\mid w=1\rangle$ satisfying conditions (C1) and (C2) has undecidable positive $AE$-theory with coefficient. In particular, its first-order theory with coefficients is undecidable.
\end{theorem}
\begin{proof}
It is an immediate consequence of Theorem \ref{thm_main}, of the fact that the $AE$-theory with coefficients of free monoids is undecidable \cite{Durnev, marchenkov} and of reducibility of theories (Proposition \ref{Diophantine_reduction}). 
\end{proof}

If we add to Theorem \ref{thm_main} the extra condition that no word in $\Delta$ starts with $a^2$, then the same result holds with all weights being $1$, i.e.\ $\vec{\lambda} = (1, \dots, 1)$. In this case  $\WeightedLengthRelation{\lambda}$ is the standard length relation $\LengthRelation$:

\begin{theorem}\label{thm_main_2}
Let $M = \pres{A}{r=1}$ and let $\Delta \subseteq A^*$ be the set of minimal invertible pieces of $r$. Suppose that: 
\begin{enumerate}
    \item[(C1)] no word from $\Delta$ is a proper subword of any other word from $\Delta$,
    \item[(C2)] there exist distinct words $\gamma, \delta \in \Delta$ with a common first letter, say $a$,
    \item[(C3)] no word in $\Delta$ starts with $a^2$.
\end{enumerate}
Then there exists a free monoid $D$ of finite rank $n\geq 2$ such that the free monoid with length relation $\FreeLengthNoFree{D}$ is interpretable in $M$ by systems of equations.
\end{theorem}

\begin{proof}
The proof works in the same way as in Theorem \ref{thm_main}, with the addition that the last part of Lemma \ref{lem_InverOfPower} now ensures that $\omega(w_i) = 1$ for all $i=1, \dots, n$. Then $\vec{\omega}= (1,  \dots, 1)$ and  $\FreeWeightedLengthNoFree{D}{\omega} = \FreeLengthNoFree{D}$, where $\FreeWeightedLengthNoFree{D}{\omega}$ is the free monoid with weighted length relation given by Theorem \ref{thm_main}. Hence $\FreeLengthNoFree{D}$ is interpretable in $M$ by systems of equations and one coefficient.
\end{proof}

We obtain an analogue of Corollary \ref{c: reduction_word_equations_weighted_length_constraints}

\begin{cor}
Let $M$ be a monoid satisfying the hypothesis of Theorem \ref{thm_main_2}. Then there exists a free monoid with (non-weighted) length relation $\FreeLengthNoFree{D}$ such that the Diophantine problem in $\FreeLengthNoFree{D}$ is reducible to the Diophantine problem in $M$. In particular, if the latter is decidable, then systems of word equations with length constraints are decidable as well.
\end{cor}

We further prove that the monoids from Theorems \ref{thm_main} and \ref{thm_main_2} naturally embed the monoids from Example \ref{r: n_bicyclic}.    

\begin{theorem}\label{thm:nbicyclicembeds}
Let $M = \pres{A}{r=1}$ and let $\Delta \subseteq A^*$ be the set of minimal invertible pieces of $r$. Suppose conditions (C1), (C2), (C3) are satisfied, i.e.: 
\begin{enumerate}
    \item[(C1)] no word from $\Delta$ is a proper subword of any other word from $\Delta$,
    \item[(C2)] there exist distinct words $\gamma, \delta \in \Delta$ with a common first letter, say $a$,
    \item[(C3)] no word in $\Delta$ starts with $a^2$.
\end{enumerate}
Let 
\[
\Sigma_a = \{ 
w \in A^* : \mbox{
$w$ is reduced and $[aw]=1$
}
\}.
\]
Then 
\begin{enumerate}
\item[(i)] $\Sigma_a$ is a finite set with $|\Sigma_a| \geq 2$; 
\item[(ii)] the submonoid of $M$ generated by $\Sigma_a$ is free with basis $\Sigma_a$. 
\end{enumerate}
Let $\Sigma_a = \{\gamma_1, \ldots, \gamma_q \}$. Then the submonoid of $M$ generated by $\{ [a] \} \cup [\Sigma_a]$ is naturally isomorphic to the  monoid defined by the presentation 
\[
\pres{a, d_1, d_2, \ldots, d_q}{
ad_1=1, \ldots, ad_q=1
}.
\]
\end{theorem}
\begin{proof}
We claim that $\Sigma_a$ is equal to the set $X = \{\beta(a^j\beta)^{-1}\mid a^j\beta\in\Delta\}$ defined above; see Lemma~\ref{lem:fin_gen_set_for_F}.    
It is immediate from the definition of $X$ that $X \subseteq \Sigma_a$. 
For the converse, let $\gamma \in \Sigma_a$. 
This means that $\gamma$ is a reduced word and $[a\gamma]=1$ in $M$.       
By Lemma~\ref{lem:fundamental:weight} we can write $a\gamma \equiv a\gamma'\gamma''$ with $\gamma' \in X$ and $a\gamma' = 1$ in $M$. 
Then $\gamma'' = (a\gamma')\gamma'' = a\gamma = 1$ in $M$. 
Since $\gamma$ is a reduced word it follows that $\gamma'' \equiv \epsilon$ and thus $\gamma \equiv \gamma' \in X$. 
This completes the proof that $X = \Sigma_a$.             

Since $X = \Sigma_a$, part (i) now follows from Lemmas~\ref{lem:fin_gen_set_for_F} and \ref{lem:basis:size}.

To prove part (ii) it will suffice to prove that $\Sigma_a = \mathcal{B}$, where $\mathcal{B}$ is the unique basis of the free monoid generated by $\Sigma_a=X$.  
To prove this it will suffice to prove that no $\gamma \in \Sigma_a$ can be written as a product of other $\gamma$ from $\Sigma_a$.     
Suppose that 
\[
\gamma = \gamma_1 \gamma_2 \ldots \gamma_m 
\]
where $\gamma_i \in \Sigma_a$ for all $1 \leq i \leq m$.   
Then
\[
1 = a \gamma = a  \gamma_1 \gamma_2 \ldots \gamma_m = \gamma_2 \ldots \gamma_m
\] 
By Lemma~\ref{lem:reduced}, $\gamma_2 \ldots \gamma_m$ is a reduced word and hence it follows that it must equal the empty word.  
Hence $m=1$ and $\gamma \equiv \gamma_1$ since they are both reduced words and they are equal in $M$.    
This completes the proof that $\Sigma_a = \mathcal{B}$, and hence completes the proof of (ii).  

For the last part, let $w \in (\{a\} \cup \Sigma_a)^* = \{a, \gamma_1, \ldots, \gamma_q \}^*$. Since $a \gamma_i = 1$ for all $i$, this word is equal in $M$ to a word $w'$ where $w'$ has the form $w' \equiv w_1 a^j$, where $w_1 \in \{\gamma_1, \ldots, \gamma_q\}^*$. We claim that in fact $\red(w) \equiv w_1 a^j$. Indeed, since none of the words appearing in the rewrite rules in \eqref{e: Zhang_presentation} ends in $a$ (because otherwise together with condition (C2) this would imply that $a$ is invertible) to show that $w_1a^j$ is reduced it suffices to prove that $w_1$ is reduced, and this was proved in Lemma~\ref{lem:reduced}. 
Therefore, each element of the submonoid of $M$ generated by $\{[a]\} \cup [\Sigma_a]$ may be uniquely written in the form $\alpha a^j$ for some $j \geq 0$ and some word $\alpha \in \{ \gamma_1, \ldots, \gamma_q \}^*$. Now consider the monoid $N$ defined by the presentation 
\[
\pres{a, d_1, d_2, \ldots, d_q}{
ad_1=1, \ldots, ad_q=1}
\] 
This is a finite complete presentation, and the reduced words are precisely those of the form $\beta a^j$ where $j \geq 0$ and $\beta \in \{d_1, \ldots, d_q \}^*$. 

Let $\phi: \{ a, d_1, \ldots, d_q \}^* \rightarrow A^*$ be the homomorphism induced by the map $a \mapsto a$, and $d_i \mapsto \gamma_i$ for $1 \leq i \leq q$. Since each relation in the presentation for $\langle a, d_1, \dots, d_q \rangle$ is preserved by this homomorphism it follows that $\phi$ induces a homomorphism $\phi: \langle a, d_1, \dots, d_q \rangle \rightarrow M$. Moreover, this homomorphism maps $\langle a, d_1, \dots, d_q \rangle$ bijectively to the submonoid of $M$ generated by $\{[a]\} \cup [\Sigma_a]$ since it clearly defines a bijection between the normal forms described above. This completes the proof of the theorem.
\end{proof}

\begin{remark}
We follow the notation of the previous Theorem \ref{thm:nbicyclicembeds}. In the proof of Theorem \ref{thm_main} we showed that both $\langle a \rangle$ and $\langle \Sigma_a\rangle$ are e-interpretable in $M$. It is natural to ask whether the submonoid $\langle a, \Sigma_a\rangle$, which by Theorem \ref{thm:nbicyclicembeds} is isomorphic to the  monoid from Example \ref{r: n_bicyclic}, is itself e-interpretable in $M$. The answer to this question is not clear and we leave it open.
\end{remark}

\subsection{One-relator monoids with hyperbolic undirected Cayley graph and hyperbolic group of units}

In this subsection we  prove some  sufficient conditions for one-relator monoids to have hyperbolic undirected Cayley graph and to have hyperbolic group of units. These are of interest to the paper given our question in the introduction regarding the reducibility of the Diophantine problem in a special one-relator monoid  to the same problem in its group of units: since the Diophantine problem in hyperbolic groups is decidable \cite{dahmani,sela_hyperbolic}, such a reduction would imply the decidability of the Diophantine problem in the one-relator monoid.

Before presenting the main result of this section we first define what we mean by the undirected Cayley graph of a monoid, and what it means for this graph to be hyperbolic.  
For more background on the theory of  
hyperbolic metric spaces 
and hyperbolic groups 
we refer the reader to \cite{BridsonHBook}. 
Let $(X,d)$ be a metric space. 
For $x, y \in X$ a \emph{geodesic path} from $x$ to $y$ is a map 
$f:[0,l] \rightarrow X$ from the closed interval $[0,l] \subseteq \mathbb{R}$ to $X$ such that 
$f(0) = x$, $f(l)=y$ and 
$d(f(a),f(b)) = |a-b|$ for all $a,b \in [0,l]$. 
Note in particular this implies that $d(x,y)=l$.
The image $\alpha$ of the map $f$ is called a \emph{geodesic segment} with endpoints $x$ and $y$. 
A \emph{geodesic metric space} is one in which there exist geodesic 
segments between all pairs of points. 
Note that in general there can be more than one geodesic segment between a given pair of points. 
A \emph{geodesic triangle} $\Omega$ in $X$ is a union of three geodesic segments from $x$ to $y$, $y$ to $z$ and $z$ to $x$, where $x,y,z \in X$. 
These three geodesic segments are called the \emph{sides} of the geodesic triangle $\Omega$.  

\begin{dfn}\label{def:hyperbolic:space}
Let $X$ be a geodesic metric space and let 
$\Omega$ be a geodesic triangle in $X$ with sides $\alpha$, $\beta$ and $\gamma$.   
The triangle $\Omega$ is called \emph{$\delta$-slim} if 
for every point $a$ on $\alpha$ the distance from $a$ to $\beta \cup \gamma$ is less than $\delta$, and similarly 
every point $b$ on $\beta$ is within distance $\delta$ of $\alpha \cup \gamma$, and   
every point $c$ on $\gamma$ is within distance $\delta$ of $\alpha \cup \beta$. 
If every geodesic triangle in $X$ is $\delta$-slim then we say that the geodesic metric space $X$ is \emph{$\delta$-hyperbolic}. 
If $X$ is $\delta$-hyperbolic for some $\delta > 0$ then we say $X$ is \emph{hyperbolic}. 
\end{dfn}

Let $M$ be a monoid generated by a set $A$. 
Then by the \emph{undirected Cayley graph $\Gamma(M,A)$ of $M$ with respect to the generating set $A$} we mean the graph with vertex set $M$   
and where there is an undirected edge connecting $m \in M$ to $n \in M$ if and only if $ma=n$ or $na=m$ for some $a \in A$. 
Note that here we have opted to work with the right Cayley graph, but the results we prove in this subsection are also true working with the left Cayley graph instead.  
The graph $\Gamma(M,A)$ is a metric space with the usual distance metric on graphs where for $a,b \in M$ we define $d(a,b)$ to be the shortest length of a path in $\Gamma(M,A)$ from $a$ to $b$. 
This is not a geodesic metric space, but can be made into one in a natural way by  
making each edge isometric to the unit interval $[0,1]$ and extending the metric to the points of these edges in the obvious way. 
This is called the \emph{geometric realisation} of the Cayley graph. 
We say that the \emph{undirected Cayley graph of a monoid $M$ is hyperbolic} if the geometric realisation of $\Gamma(M,A)$ 
is a hyperbolic metric space. 
If $M$ is a finitely generated monoid, it may be shown that this 
property is independent of the choice of finite generating set for $M$, so it makes sense to talk about a finitely generated monoid having a hyperbolic undirected Cayley graph, without reference to any specific finite generating set.   

\begin{prop}\label{prop:hyperbolic:from:units} 
Let $M = \langle A \mid w=1 \rangle$. Let $G$ be the group of units of $M$. If $G$ is a hyperbolic group then the undirected Cayley graph of $M$ is hyperbolic. 
\end{prop}
\begin{proof} 
As usual, let $\Delta$ be the set of minimal invertible pieces of the relator $w$ (see the discussion above Theorem \ref{thm:Zhang} for further details). 
Let $I$ be the set of all non-empty prefixes of the words from $\Delta$, that is  
\[
I = \{x \in A^+ \ \mid \ xy \in \Delta \ \mbox{for some} \ y \in A^* \}.
\]
Let $Y = \{ [u] : u \in I \}$. 
Then, by Zhang \cite[Lemma~3.3]{Zhang1992}, $Y$ is a finite generating set for the submonoid of right units $R$ of $M$. 
Note that $R$ is the $\mathcal{R}$-class of the identity element of $M$, where $\mathcal{R}$ is Green's $\mathcal{R}$-relation on $M$ defined by saying $m \mathcal{R} n$ if and only if $mM = nM$. 
Clearly $\Delta$ is a subset of $I$. 
Let $\mathcal{G}$ be the 
underlying undirected graph of the 
right Cayley graph of the monoid $R$, with respect to the generating set $Y$.  
So $\mathcal{G}$ has vertex set $R$ and edges $\{[u], [ux]\}$ where $u \in A^*$, $x \in I$ and $\{[u], [ux]\}$ is a subset of $R$. 
Note that $\mathcal{G}$ is a connected infinite graph
and its vertices have  
bounded degree since 
$R$ is a right cancellative monoid. 
We use $\mathcal{S}$ to denote the undirected Sch\"{u}tzenberger graph of the $\mathcal{R}$-class $R$. So $\mathcal{S}$ also has vertex set $R$ but has edges $\{[u], [ua]\}$ where $u \in A^*$, $a \in A$ and $\{[u], [ua]\}$ is a subset of $R$. 

We claim that the identity map on $R$ defines a quasi-isometry between the graph $\mathcal{G}$ and the graph $\mathcal{S}$. 

To prove this claim, let $d_\mathcal{S}$ and $d_\mathcal{G}$ denote the distances in each of these graphs. Consider an arbitrary edge $\{[u], [ux]\}$ in the graph $\mathcal{G}$. Let $D$ be the maximum length of a word in $\Delta$. Then $d_{\mathcal{S}}([u],[ux]) \leq D$.  For the converse, let $\{[u], [ua] \}$ be an arbitrary edge in the graph $\mathcal{S}$. We claim that $d_\mathcal{G}([u],[ua]) \leq 2$. We may assume without loss of generality that $u$ is a reduced word. There are now two cases to consider. 

First suppose that $ua$ is a reduced word. It then follows from 
\cite[Lemma~3.3]{Zhang1992}
that $ua \in I^*$ (i.e. is a graphical product of words from $I$). Note that $ua$ may admit several different decompositions in $I^*$. Write $ua = u' \gamma$ where $\gamma \in I$ and $u' \in I^*$. If $|\gamma|=1$ then $a=\gamma \in I$ and so $d_\mathcal{G}([u],[ua])=1$. Now suppose that $|\gamma|>1$.  Write $\gamma = \gamma'a$ with $\gamma' \in I$.  Then we have $u=u'\gamma'$ and both $\{[u'], [u'\gamma'] \}$ and $\{[u'], [u'\gamma] \}$ are  edges in the graph $\mathcal{G}$. It follows that $d_\mathcal{G}([u],[ua])  = d_\mathcal{G}([u'\gamma'], [u'\gamma]) \leq 2$.        

Now suppose that $ua$ is not a reduced word. Since $u$ is reduced, it follows that we can write $ua = u'\gamma$ where $\gamma \in \Delta$ is a non-empty word. Then arguing as in the previous paragraph, either $|\gamma|=1$ and $d_{\mathcal{G}}([u],[ua])=1$, or else $|\gamma|>1$ and $d_\mathcal{G}([u],[ua])  \leq 2$. This completes the proof of the claim that the identity mapping on $R$ induces a quasi-isometry between the graph $\mathcal{G}$ and the graph $\mathcal{S}$. 

It follows from 
\cite[Theorem~4.5]{Zhang1992}
that the submonoid of right units $R$ of $M$ is isomorphic to a monoid free product $T \ast G$ where $T$ is a free monoid of finite rank, and $G$ is the group of units of the monoid $M$. Since the Cayley graph of a free monoid is a tree, it then follows that the undirected Cayley graph $\mathcal{G}$ of $R \cong T \ast G$  is  hyperbolic. Since $\mathcal{S}$ is quasi-isometric to $\mathcal{G}$ we conclude that the undirected Sch\"{u}tzenberger graph of the $\mathcal{R}$-class of the identity element is hyperbolic.  

It follows from the results in 
\cite[Section~3]{
GraySteinbergPaper2}
that (i) the Sch\"{u}tzenberger graphs of any pair of $\mc{R}$-classes of $M$ are isomorphic to each other, and (ii) for every $\mathcal{R}$-class $R'$ of $M$ there is at most one edge $\{m, ma\}$ in the Cayley graph of $M$ such that $m \in M$, $a \in A$, with $ma \in R'$ but $m \not\in R'$, and (iii) the quotient graph with vertex set the $\mathcal{R}$-classes of $M$ and edges all edges $\{m, n\}$ from the Cayley graph of $M$ such that $(m, n) \not\in \mathcal{R}$ is a rooted tree. 

Combining these observations we see that the Cayley graph of $M$ has the structure of a ``regular tree of copies of'' the hyperbolic graph $\mathcal{S}$. From this it then quickly follows 
(e.g. by applying \cite[Theorem 5.4]{gray2019algorithmic}) 
that the undirected Cayley graph of $M$ is hyperbolic.  
\end{proof}

In fact, using a similar argument, it may be shown that Proposition~\ref{prop:hyperbolic:from:units} holds more generally for any finitely presented monoid $M$ defined by a presentation of the form
\[
\langle A \mid w_1=1, \ldots, w_k=1 \rangle.
\]

\begin{prop}\label{prop_hyperbolic}
Let $M = \pres{A}{w^k=1}$  $(k \geq 2)$. 
Then the group of units of $M$ is a one-relator group with torsion.
It follows that the group of units of $M$ is a hyperbolic group, and the undirected Cayley graph of $M$ is a hyperbolic  
metric space. 
\end{prop}
\begin{proof}
It follows from results of Adjan \cite{Adjan1966} that %
the group of units $G$ of $M$ is a one-relator group with torsion (see \cite[Section~3]{GraySteinbergPaper2} for a proof of this). 
By the Newman Spelling Theorem 
\cite[Theorem~5.5]{lyndon_schupp}
we  have that $G$ is a hyperbolic group. This and Proposition~\ref{prop:hyperbolic:from:units} imply that the undirected Cayley graph of $M$ is hyperbolic. %
\end{proof}

\section{Applications, examples and open problems}
\label{sec:final:section:applications:open:problems}

In this section we list some examples, and classes of examples, of monoids to which the main results of this paper apply. We shall also collect together a selection of open problems, and possible future research directions, which naturally arise from our results. As part of this we will identify the simplest examples of one-relator monoids for which we do not yet know whether or not the Diophantine problem is decidable.  
In general, we do not know if there is an example of a one-relator monoid of the form $\pres{A}{r=1}$ with undecidable Diophantine problem. 

Let us begin by recording some examples of one-relator monoids of the form $\pres{A}{r=1}$ where we have shown that the Diophantine problem is decidable. Consider, in particular the case of $2$-generated one-relator monoids $\pres{a,b}{r=1}$. Let $M$ denote the monoid defined by this presentation.  Very often questions about one-relator monoids can be reduced to just considering the $2$-generator case e.g. this is the case for the word problem. 

By Makanin \cite{makanin2} the Diophantine problem is decidable for the free monoid $\pres{a,b}{}$, while in \cite[Example 21]{Diekert} it is proved that it is decidable for the bicyclic monoid $\pres{a,b}{ab=1}$.  Now consider the general case $\pres{a,b}{r=1}$ and let $r = r_1 r_2 \ldots r_k$ be the decomposition of $r$ into minimal invertible pieces as described in Section~\ref{sec:DP:for:one-relator} and in Example~\ref{ex:adjan:example} and the paragraph preceding it.  
If $r \in \{a\}^*$ or $r \in \{b\}^*$ then the monoid is a free product of a free monoid of rank one and a finite cyclic group, and thus the Diophantine problem is decidable by \cite{Diekert}.  
Now suppose that both the letters $a$ and $b$ appear in the defining relator $r$.  There are then two cases to consider.  If there are minimal invertible pieces $r_i$ and $r_j$ such that the first letter of $r_i$ equals the last letter of $r_j$, then 
applying the Adjan overlap algorithm 
it follows that both $a$ and $b$ both represent invertible elements of $M$ and hence    $M$ is a group.  
In this case, $M$ is the group defined by the same one-relator group presentation, and hence $M$ is a so-called \emph{positive} one-relator group.  Such groups have been studied e.g. by Baumslag \cite{Baumslag1971} and Wise \cite{Wise2001}.  This motivates the question of whether the Diophantine problem is decidable for positive one-relator groups. Up to symmetry the case that remains is when all the invertible pieces $r_i$ $(1 \leq i \leq k)$ begin with the letter $a$ and end with the letter $b$.   This case then divides into two subcases, either (i) all of the pieces $r_i$ are equal to each other as words, or (ii) there is some pair of minimal invertible pieces $r_i$ and $r_j$ with $r_i \not\equiv r_j$.  Note that subcase (i) includes in particular the case where there is a single invertible piece.  This is precisely the case where the relator $r$ is self-overlap free meaning that no proper non-empty prefix is equal to a proper non-empty suffix of $r$.  This in turn is equivalent to saying that the group of units of the monoid is the trivial group.   Also note that many of the examples in (ii) will satisfy the conditions (C1) and (C2) (and (C3)) from Section~\ref{sec:DP:for:one-relator},  and thus  the main theorems of that section,  Theorem~\ref{thm_main} and Theorem~\ref{thm_main_2}, will apply to them. Some examples of these are listed in the introduction after Theorem \ref{t: main_thm_intro}.

A similar division into cases can also be done for one-relator monoids $\pres{A}{r=1}$ with more than two generators.  
For instance, 
as we already saw in Example~\ref{ex:Z}, 
the monoid $\langle a,b,c,d \mid aba=1 \rangle$ has decidable Diophantine problem by Theorem~\ref{thm:free:prod:application} above, since all the letters in the relator are invertible, and the group of units is the infinite cyclic group which has decidable Diophantine problem.   
Similarly the monoid  $\pres{a,b,c,d,e,f}{abcddcbbaa=1}$ has decidable Diophantine problem, again applying Theorem~\ref{thm:free:prod:application}.
Indeed, applying the Adjan overlap algorithm we deduce that all the letters $a$, $b$, $c$ and $d$ appearing in the defining relator are invertible, and the group of units of this monoid is defined by the group presentation   
\[\mathrm{Gp}\langle a,b,c,d \mid 
abcddcbbaa=1 \rangle.\]
To apply 
Theorem~\ref{thm:free:prod:application}
we need to show that this group has decidable Diophantine problem. 
To show this, note that this group can be written 
\[\mathrm{Gp}\langle a,b,c,d \mid cddc = b^{-1} a^{-1} a^{-1} a^{-1} b^{-1} b^{-1} \rangle.\]
The words $cddc$ and $b^{-1} a^{-1} a^{-1} a^{-1} b^{-1} b^{-1}$ are 
non-primitive since the words $cddc$ and $bbaaab$ are not Christoffel words (see e.g. \cite{Reutenauer2019}), and neither are any of the conjugates of these words, since the first word have the same number of $c$s and $d$s, and similarly for the second word. It is known, see \cite{Kharlampovich1998, Juhasz1994, Bestvina1992},
that a cyclically pinched one-relator group defined by a presentation $\mathrm{Gp}\langle A | u=v \rangle$, where $u$ and $v$ are non-primitive words written over disjoint sets of letters, and it is not the case that both $u$ and $v$ are proper powers, is hyperbolic. 
Hence the group of units of  $\pres{a,b,c,d,e,f}{abcdcbba=1}$ is a hyperbolic group and thus by
Theorem~\ref{thm:free:prod:application} above this monoid has decidable Diophantine problem. 
Many other examples similar to this can be written down. 
This gives a reasonably rich source of examples of one-relator monoids $\pres{A}{r=1}$ which have solvable Diophantine problem as a consequence of the fact that their groups of units are hyperbolic. 
We do not know in general whether having a hyperbolic group of units is enough to imply that a one-relator monoid of the form $\pres{A}{r=1}$ has solvable Diophantine problem. 
As explained in the introduction, this was one of the original motivating questions for the work done in this paper. By Proposition \ref{prop_hyperbolic} and Theorem \ref{thm_main_2}, a positive answer to this questions implies decidability of word equations with length constraints.

In light of this discussion, it is sensible to identify the simplest examples of one-relator monoids of the form $\pres{A}{r=1}$ for which we neither know that the Diophantine problem is decidable, but we also do not know of a reduction theorem (like the theorems from Section~\ref{sec:DP:for:one-relator} above) of a known difficult open problem. 
Thus we ask whether either of the monoids  $\pres{b,c}{b^2c=1}$ or $\pres{a,b,c}{abc=1}$ has decidable  Diophantine problem?
Initial investigations indicate that this might relate to solving word equations with a variation on the notion of twisting, in the sense of \cite{Diekert2017a}.
More generally we ask the following 
\begin{question}\label{q: first_question}
If the word $w \in A^*$ has no self overlaps,
i.e.\ there is no non-empty word which is both a proper prefix of $w$ and a proper suffix of $w$,  
then is the Diophantine problem 
for the one-relator monoid
$
\langle A \mid w=1 \rangle
$
decidable?
\end{question}
Note that the condition that $w$ has no self overlaps is equivalent to saying the group of units of this monoid is trivial (this follows from the discussion immediately before the statement of Theorem~\ref{thm:Zhang}).  

The corresponding class of monoids with torsion are 
also
not covered by any of the theorems in this paper. 
Thus we ask whether 
$\pres{b,c}{bcbc=1}$ 
has decidable Diophantine problem? 
More generally, of course, we can ask whether the Diophantine problem is decidable for monoids $\pres{A}{w^n=1}$ where $w$ has no self overlaps.  

Finally, we restate some natural questions which
have arisen in this work. 
As already mentioned above, if any of these problems has a positive answer, then as a corollary this would give a positive solution to the 
open problem of solving word equations with length constraints. 

\begin{question}
Is the Diophantine problem decidable for one-relator monoids 
of the form $\pres{A}{w^n=1}$ where $n>1$?    
\end{question}

\begin{question}  \label{q: last_question}
Let $M$ be the monoid defined by $\pres{A}{w=1}$   
and let $G$ be the group of units of $M$. 
If the Diophantine problem is decidable in G, then does it follow that it is decidable in $M$?  
\end{question}

It follows from the results in the present paper that the positive $AE$-theory is in general undecidable in  the classes of monoids from Questions \ref{q: first_question} through Question \ref{q: last_question} (due to Theorem \ref{t: undec_AE_one_relator}, Proposition \ref{prop_hyperbolic}, and Remark \ref{r: n_bicyclic}).

\section{Acknowledgements}

\noindent We would like to thank  the anonymous reviewer who carefully read our paper, and whose comments and suggestions helped improve it.

\medskip

\noindent We are also thankful to Olga Kharlampovich and Alexei Miasnikov for useful conversations.

\medskip

\noindent This research was initiated during the Heilbronn Focused Research Workshop 
``Equations in One-Relator Groups and Semigroups’’, in September 2018, at the ICMS in Edinburgh 
organised by Laura Ciobanu of Heriot-Watt University.

\medskip

\noindent  The first named author was supported by the ERC grant PCG-336983, by the Basque Government grant IT974-16, and by the MINECO Grant PID2019-107444GA-I00.

\bibliography{bib.bib}

\def\cprime{$'$}
\begin{thebibliography}{10}

\bibitem{norn}
P.~A. Abdulla, M.~F. Atig, Y.-F. Chen, L.~Hol{\'i}k, A.~Rezine, P.~R{\"u}mmer,
  and J.~Stenman.
\newblock Norn: An {SMT} solver for string constraints.
\newblock In D.~Kroening and C.~S. P{\u{a}}s{\u{a}}reanu, editors, {\em
  Computer Aided Verification}, pages 462--469, Cham, 2015. Springer
  International Publishing.

\bibitem{Adjan1966}
S.~I. Adjan.
\newblock Defining relations and algorithmic problems for groups and
  semigroups.
\newblock {\em Trudy Mat. Inst. Steklov.}, 85:123, 1966.

\bibitem{Araujo2018}
J.~Ara\'{u}jo, M.~Kinyon, J.~Konieczny, and A.~Malheiro.
\newblock Decidability and independence of conjugacy problems in finitely
  presented monoids.
\newblock {\em Theoret. Comput. Sci.}, 731:88--98, 2018.

\bibitem{Araujo2014}
J.~Ara\'{u}jo, J.~Konieczny, and A.~Malheiro.
\newblock Conjugation in semigroups.
\newblock {\em J. Algebra}, 403:93--134, 2014.

\bibitem{Arora}
S.~Arora and B.~Barak.
\newblock {\em Computational Complexity: A Modern Approach}.
\newblock Cambridge University Press, New York, NY, USA, 1st edition, 2009.

\bibitem{ABC}
A.~Aydin, L.~Bang, and T.~Bultan.
\newblock Automata-based model counting for string constraints.
\newblock In D.~Kroening and C.~S. P{\u{a}}s{\u{a}}reanu, editors, {\em
  Computer Aided Verification}, pages 255--272, Cham, 2015. Springer
  International Publishing.

\bibitem{cvc4}
C.~Barrett, C.~L. Conway, M.~Deters, L.~Hadarean, D.~Jovanovi{\'{c}}, T.~King,
  A.~Reynolds, and C.~Tinelli.
\newblock Cvc4.
\newblock In G.~Gopalakrishnan and S.~Qadeer, editors, {\em Computer Aided
  Verification}, pages 171--177, Berlin, Heidelberg, 2011. Springer Berlin
  Heidelberg.

\bibitem{SMT_Barrett}
C.~Barrett, R.~Sebastiani, S.~Seshia, and C.~Tinelli.
\newblock {\em Satisfiability modulo theories}, volume 185 of {\em Frontiers in
  Artificial Intelligence and Applications}, pages 825--885.
\newblock 1 edition, 2009.

\bibitem{Baumslag1971}
G.~Baumslag.
\newblock Positive one-relator groups.
\newblock {\em Trans. Amer. Math. Soc.}, 156:165--183, 1971.

\bibitem{z3str3}
M.~{Berzish}, V.~{Ganesh}, and Y.~{Zheng}.
\newblock Z3str3: A string solver with theory-aware heuristics.
\newblock In {\em 2017 Formal Methods in Computer Aided Design (FMCAD)}, pages
  55--59, Oct 2017.

\bibitem{Bestvina1992}
M.~Bestvina and M.~Feighn.
\newblock A combination theorem for negatively curved groups.
\newblock {\em J. Differential Geom.}, 35(1):85--101, 1992.

\bibitem{BridsonHBook}
M.~R. Bridson and A.~Haefliger.
\newblock {\em Metric spaces of non-positive curvature}, volume 319 of {\em
  Grundlehren der Mathematischen Wissenschaften [Fundamental Principles of
  Mathematical Sciences]}.
\newblock Springer-Verlag, Berlin, 1999.

\bibitem{Buchi}
J.~R. B\"{u}chi and S.~Senger.
\newblock Definability in the existential theory of concatenation and
  undecidable extensions of this theory.
\newblock {\em Z. Math. Logik Grundlag. Math.}, 34(4):337--342, 1988.

\bibitem{Casals2}
M.~Casals-Ruiz and I.~Kazachkov.
\newblock On systems of equations over free products of groups.
\newblock {\em Journal of Algebra}, 333(1):368 -- 426, 2011.

\bibitem{EDT0L}
L.~{Ciobanu Radomirovic}, V.~Volker~Diekert, and M.~Elder.
\newblock Solution sets for equations over free groups are {EDT0L} languages.
\newblock {\em International Journal of Algebra and Computation},
  26(5):843--886, 8 2016.

\bibitem{dahmani}
F.~Dahmani and V.~Guirardel.
\newblock Foliations for solving equations in groups: free, virtually free, and
  hyperbolic groups.
\newblock {\em Journal of Topology}, 3(2):343--404, 2016.

\bibitem{ganesh}
J.~D. Day, V.~Ganesh, P.~He, F.~Manea, and D.~Nowotka.
\newblock The satisfiability of word equations: Decidable and undecidable
  theories.
\newblock In I.~Potapov and P.-A. Reynier, editors, {\em Reachability
  Problems}, pages 15--29, Cham, 2018. Springer International Publishing.

\bibitem{ganesh_18}
J.~D. Day, V.~Ganesh, P.~He, F.~Manea, and D.~Nowotka.
\newblock The satisfiability of word equations: decidable and undecidable
  theories.
\newblock In {\em Reachability problems}, volume 11123 of {\em Lecture Notes in
  Comput. Sci.}, pages 15--29. Springer, Cham, 2018.

\bibitem{SMT_deMoura}
L.~De~Moura and N.~Bj{\o}rner.
\newblock Satisfiability modulo theories: Introduction and applications.
\newblock {\em Commun. ACM}, 54(9):69--77, September 2011.

\bibitem{Deis2007}
T.~Deis, J.~Meakin, and G.~S\'{e}nizergues.
\newblock Equations in free inverse monoids.
\newblock {\em Internat. J. Algebra Comput.}, 17(4):761--795, 2007.

\bibitem{Diekert}
V.~Diekert and M.~Lohrey.
\newblock Word equations over graph products.
\newblock {\em International Journal of Algebra and Computation},
  18(03):493--533, 2008.

\bibitem{Diekert_1700}
V.~Diekert.
\newblock More than 1700 years of word equations.
\newblock In A.~Maletti, editor, {\em Algebraic Informatics}, pages 22--28,
  Cham, 2015. Springer International Publishing.

\bibitem{Diekert2017a}
V.~Diekert and M.~Elder.
\newblock Solutions to twisted word equations and equations in virtually free
  groups.
\newblock {\em International Journal of Algebra and Computation},
  30(04):731--819, 2020.

\bibitem{Diekert2017b}
V.~Diekert, F.~Martin, G.~S\'{e}nizergues, and P.~V. Silva.
\newblock Equations over free inverse monoids with idempotent variables.
\newblock {\em Theory Comput. Syst.}, 61(2):494--520, 2017.

\bibitem{Durnev}
V.~G. Durnev.
\newblock Undecidability of the positive $\forall\exists^3$-theory of a free
  semigroup.
\newblock {\em Siberian Mathematical Journal}, 36(5):917--929, Sep 1995.

\bibitem{hampi}
V.~Ganesh, A.~Kie\.{z}un, S.~Artzi, P.~J. Guo, P.~Hooimeijer, and M.~Ernst.
\newblock H{AMPI}: a string solver for testing, analysis and vulnerability
  detection.
\newblock In {\em Computer aided verification}, volume 6806 of {\em Lecture
  Notes in Comput. Sci.}, pages 1--19. Springer, Heidelberg, 2011.

\bibitem{length_ctrts_2}
V.~Ganesh, M.~Minnes, A.~Solar-Lezama, and M.~Rinard.
\newblock Word equations with length constraints: What's decidable?
\newblock In A.~Biere, A.~Nahir, and T.~Vos, editors, {\em Hardware and
  Software: Verification and Testing}, pages 209--226, Berlin, Heidelberg,
  2013. Springer Berlin Heidelberg.

\bibitem{GMO_rings}
A.~{Garreta}, A.~{Miasnikov}, and D.~{Ovchinnikov}.
\newblock {Diophantine problems in rings and algebras: undecidability and
  reductions to rings of algebraic integers}.
\newblock {\em arXiv e-prints}, May 2018.

\bibitem{GMO}
A.~Garreta, A.~Miasnikov, and D.~Ovchinnikov.
\newblock Diophantine problems in solvable groups.
\newblock {\em Bulletin of Mathematical Sciences}, 10(01):2050005, 2020.

\bibitem{gray2019algorithmic}
R.~D. Gray, P.~V. Silva, and N.~Szakács.
\newblock Algorithmic properties of inverse monoids with hyperbolic and
  tree-like sch\"utzenberger graphs.
\newblock {\em arXiv e-prints}, 2019.

\bibitem{GraySteinbergPaper2}
R.~D. {Gray} and B.~{Steinberg}.
\newblock {Topological finiteness properties of monoids. Part 2: special
  monoids, one-relator monoids, amalgamated free products, and HNN extensions}.
\newblock {\em arXiv e-prints}, May 2018.

\bibitem{Hodges}
W.~Hodges, S.~Hodges, H.~Wilfrid, G.~Rota, B.~Doran, P.~Flajolet, T.~Lam,
  E.~Lutwak, and M.~Ismail.
\newblock {\em Model Theory}.
\newblock Encyclopedia of Mathematics and its Applications. Cambridge
  University Press, 1993.

\bibitem{Howie}
J.~M. Howie.
\newblock {\em Fundamentals of semigroup theory}, volume~12 of {\em London
  Mathematical Society Monographs. New Series}.
\newblock The Clarendon Press, Oxford University Press, New York, 1995.
\newblock Oxford Science Publications.

\bibitem{jez_linear_space}
A.~Jez.
\newblock {Word Equations in Nondeterministic Linear Space}.
\newblock In I.~Chatzigiannakis, P.~Indyk, F.~Kuhn, and A.~Muscholl, editors,
  {\em 44th International Colloquium on Automata, Languages, and Programming
  (ICALP 2017)}, volume~80 of {\em Leibniz International Proceedings in
  Informatics (LIPIcs)}, pages 95:1--95:13, Dagstuhl, Germany, 2017. Schloss
  Dagstuhl--Leibniz-Zentrum fuer Informatik.

\bibitem{Juhasz1994}
A.~Juh\'{a}sz and G.~Rosenberger.
\newblock On the combinatorial curvature of groups of {$F$}-type and other
  one-relator free products.
\newblock In {\em The mathematical legacy of {W}ilhelm {M}agnus: groups,
  geometry and special functions ({B}rooklyn, {NY}, 1992)}, volume 169 of {\em
  Contemp. Math.}, pages 373--384. Amer. Math. Soc., Providence, RI, 1994.

\bibitem{KhMi_icm}
O.~{Kharlampovich} and A.~G. {Miasnikov}.
\newblock Model theory and algebraic geometry in groups, non-standard actions
  and algorithmic problems.
\newblock {\em Proceedings of the Intern. Congress of Mathematicians, Seoul, v.
  2, invited lectures, 223-244}, 2014.

\bibitem{Kharlampovich1998}
O.~Kharlampovich and A.~Myasnikov.
\newblock Hyperbolic groups and free constructions.
\newblock {\em Trans. Amer. Math. Soc.}, 350(2):571--613, 1998.

\bibitem{KharlampovichLopez}
O.~Kharlampovich and L.~L{\'o}pez.
\newblock Bi-interpretability of some monoids with the arithmetic and
  applications.
\newblock {\em Semigroup Forum}, 99(1):126--139, Aug 2019.

\bibitem{Kharlampovich2019}
O.~Kharlampovich, L.~Lopez~Cruz, and A.~Miasnikov.
\newblock The {D}iophantine problem in some metabelian groups.
\newblock {\em Mathematics of Computation}, 89:1, 02 2020.

\bibitem{Kha_Mia_tarsi}
O.~Kharlampovich and A.~Myasnikov.
\newblock Elementary theory of free non-abelian groups.
\newblock {\em J. Algebra}, 302(2):451--552, 2006.

\bibitem{KhMi_rio}
O.~Kharlampovich and A.~G. Myasnikov.
\newblock Equations and fully residually free groups.
\newblock In O.~Bogopolski, I.~Bumagin, O.~Kharlampovich, and E.~Ventura,
  editors, {\em Combinatorial and Geometric Group Theory}, pages 203--242,
  Basel, 2010. Birkh{\"a}user Basel.

\bibitem{lallement}
G.~Lallement.
\newblock On monoids presented by a single relation.
\newblock {\em J. Algebra}, 32:370--388, 1974.

\bibitem{length_ctrts}
A.~W. Lin and R.~Majumdar.
\newblock Quadratic word equations with length constraints, counter systems,
  and presburger arithmetic with divisibility.
\newblock In S.~K. Lahiri and C.~Wang, editors, {\em Automated Technology for
  Verification and Analysis}, pages 352--369, Cham, 2018. Springer
  International Publishing.

\bibitem{lohrey}
M.~Lohrey and G.~S{\'e}nizergues.
\newblock Theories of {HNN}-extensions and amalgamated products.
\newblock In M.~Bugliesi, B.~Preneel, V.~Sassone, and I.~Wegener, editors, {\em
  Automata, Languages and Programming}, pages 504--515, Berlin, Heidelberg,
  2006. Springer Berlin Heidelberg.

\bibitem{lothaire}
M.~Lothaire.
\newblock {\em Combinatorics on Words}.
\newblock Cambridge Mathematical Library. Cambridge University Press, 2
  edition, 1997.

\bibitem{lyndon_schupp}
R.~Lyndon and P.~Schupp.
\newblock {\em Combinatorial Group Theory}.
\newblock Classics in Mathematics. Springer Berlin Heidelberg, 2015.

\bibitem{Makanin1966}
G.~S. Makanin.
\newblock On the identity problem in finitely defined semigroups.
\newblock {\em Dokl. Akad. Nauk SSSR}, 171:285--287, 1966.

\bibitem{makanin2}
G.~S. Makanin.
\newblock The problem of solvability of equations in a free semigroup.
\newblock {\em Mathematics of the {USSR}-Sbornik}, 32(2):129--198, feb 1977.

\bibitem{Makanin}
G.~S. Makanin.
\newblock Equations in a free group.
\newblock {\em Mathematics of the USSR-Izvestiya}, 21(3):483, 1983.

\bibitem{marchenkov}
S.~S. Marchenkov.
\newblock Undecidability of the positive {$\forall \exists $}-theory of a free
  semigroup.
\newblock {\em Sibirsk. Mat. Zh.}, 23(1):196--198, 223, 1982.

\bibitem{Marker}
D.~Marker.
\newblock {\em Model Theory : An Introduction}.
\newblock Graduate Texts in Mathematics. Springer New York, 2013.

\bibitem{matiasevich_H10}
J.~V. Matijasevi\v{c}.
\newblock The {D}iophantineness of enumerable sets.
\newblock {\em Dokl. Akad. Nauk SSSR}, 191:279--282, 1970.

\bibitem{matyasevich}
Y.~Matiyasevich.
\newblock Word equations, {F}ibonacci numbers, and {H}ilbert's {T}enth
  {P}roblem.
\newblock 02 2019.

\bibitem{Narendran85}
P.~Narendran and F.~Otto.
\newblock Complexity results on the conjugacy problem for monoids.
\newblock {\em Theoret. Comput. Sci.}, 35(2-3):227--243, 1985.

\bibitem{Narendran1984}
P.~Narendran, F.~Otto, and K.~Winklmann.
\newblock The uniform conjugacy problem for finite {C}hurch-{R}osser {T}hue
  systems is {NP}-complete.
\newblock {\em Inform. and Control}, 63(1-2):58--66, 1984.

\bibitem{Otto1984}
F.~Otto.
\newblock Conjugacy in monoids with a special {C}hurch-{R}osser presentation is
  decidable.
\newblock {\em Semigroup Forum}, 29(1-2):223--240, 1984.

\bibitem{Plandowski2004}
W.~Plandowski.
\newblock Satisfiability of word equations with constants is in {PSPACE}.
\newblock {\em J. ACM}, 51(3):483--496, 2004.

\bibitem{quine}
W.~V. Quine.
\newblock Concatenation as a basis for arithmetic.
\newblock {\em Journal of Symbolic Logic}, 11(4):105–114, 1946.

\bibitem{Razborov1}
A.~A. Razborov.
\newblock On systems of equations in a free group.
\newblock {\em Mathematics of the USSR-Izvestiya}, 25(1):115, 1985.

\bibitem{Reutenauer2019}
C.~Reutenauer.
\newblock {\em From {C}hristoffel words to {M}arkoff numbers}.
\newblock Oxford University Press, Oxford, 2019.

\bibitem{Rips1995}
E.~Rips and Z.~Sela.
\newblock Canonical representatives and equations in hyperbolic groups.
\newblock {\em Inventiones mathematicae}, 120(1):489--512, 1995.

\bibitem{Romankov_survey}
V.~A. Roman{\cprime}kov.
\newblock Equations over groups.
\newblock {\em Groups Complexity Cryptology}, 4:191--239, 2012.

\bibitem{Rozenblat1985}
B.~V. Rozenblat.
\newblock Diophantine theories of free inverse semigroups.
\newblock {\em Sibirsk. Mat. Zh.}, 26(6):101--107, 190, 1985.

\bibitem{Sela_tarski}
Z.~Sela.
\newblock Diophantine geometry over groups. {VI}. {T}he elementary theory of a
  free group.
\newblock {\em Geom. Funct. Anal.}, 16(3):707--730, 2006.

\bibitem{sela_hyperbolic}
Z.~Sela.
\newblock Diophantine geometry over groups. {VII}. {T}he elementary theory of a
  hyperbolic group.
\newblock {\em Proc. Lond. Math. Soc. (3)}, 99(1):217--273, 2009.

\bibitem{Sela2016}
Z.~{Sela}.
\newblock {Word Equations I: Pairs and their Makanin-Razborov diagrams}.
\newblock {\em arXiv e-prints}, Jul 2016.

\bibitem{S3P}
M.-T. Trinh, D.-H. Chu, and J.~Jaffar.
\newblock Progressive reasoning over recursively-defined strings.
\newblock In S.~Chaudhuri and A.~Farzan, editors, {\em Computer Aided
  Verification}, pages 218--240, Cham, 2016. Springer International Publishing.

\bibitem{Wise2001}
D.~T. Wise.
\newblock The residual finiteness of positive one-relator groups.
\newblock {\em Comment. Math. Helv.}, 76(2):314--338, 2001.

\bibitem{stranger}
F.~Yu, M.~Alkhalaf, and T.~Bultan.
\newblock Stranger: An automata-based string analysis tool for php.
\newblock In J.~Esparza and R.~Majumdar, editors, {\em Tools and Algorithms for
  the Construction and Analysis of Systems}, pages 154--157, Berlin,
  Heidelberg, 2010. Springer Berlin Heidelberg.

\bibitem{Zhang1991}
L.~Zhang.
\newblock Conjugacy in special monoids.
\newblock {\em J. Algebra}, 143(2):487--497, 1991.

\bibitem{Zhang1992}
L.~Zhang.
\newblock Applying rewriting methods to special monoids.
\newblock {\em Math. Proc. Cambridge Philos. Soc.}, 112(3):495--505, 1992.

\end{thebibliography}

\end{document}